\documentclass[a4paper]{scrartcl}
\usepackage[T1]{fontenc}
\usepackage{amsmath}
\usepackage{amssymb,bbm}
\usepackage{amsthm}
\usepackage{enumerate}
\usepackage[usenames,dvipsnames]{color}
\usepackage{tikz}

\newcommand{\Z}{\mathbb{Z}}

\newcommand{\C}{\mathbb{C}}
\newcommand{\K}{\mathbbm{k}}

\newcommand{\vect}[1]{#1}
\newcommand{\TGWdat}{(R,\sigma,\vect{t},\mu)}

\newcommand{\nn}{\underline{n}}
\newcommand{\si}{\sigma}\newcommand{\al}{\alpha}
\newcommand{\la}{\lambda}\newcommand{\ga}{\gamma}
\newcommand{\ep}{\varepsilon}
\DeclareMathOperator{\grRad}{grRad}
\DeclareMathOperator{\diag}{diag}
\DeclareMathOperator{\identity}{id}
\DeclareMathOperator{\Aut}{Aut}
\DeclareMathOperator{\Rad}{Rad}
\DeclareMathOperator{\reg}{reg}
\DeclareMathOperator{\Supp}{Supp}
\DeclareMathOperator{\ad}{ad}

\DeclareMathOperator{\Spec}{Spec}

\renewcommand{\emptyset}{\varnothing}

\theoremstyle{plain}
\newtheorem{theorem}{Theorem}[section]
\newtheorem{lemma}[theorem]{Lemma}
\newtheorem{prop}[theorem]{Proposition}
\newtheorem{corollary}[theorem]{Corollary}
\theoremstyle{definition}
\newtheorem{definition}[theorem]{Definition}

\newtheorem{exmp}[theorem]{Example}
\newtheorem{prob}{Problem}
\newtheorem{remark}[theorem]{Remark}

\newtheorem*{thmA}{Theorem A}
\newtheorem*{thmB}{Theorem B}
\newtheorem*{thmC}{Theorem C}

\begin{document}

\author{Jonas T. Hartwig\footnote{E-mail: jonas.hartwig@gmail.com} \thanks{Partially supported by the Netherlands Organization for Scientific Research in the VIDI-project ``Symmetry and modularity in exactly solvable models''
 and by postdoctoral grants from the Swedish Research Council and from the S\~{a}o Paulo Research Foundation FAPESP (2008/10688-1).} 
\and Johan \"Oinert\footnote{Address: Department of Mathematical Sciences, University of Copenhagen, Universitetsparken 5, DK-2100 Copenhagen \O, Denmark, E-mail: oinert@math.ku.dk} \thanks{Partially supported by The Swedish Research Council, The
Swedish Foundation for International Cooperation in Research and Higher Education
(STINT), The Crafoord Foundation, The Royal Physiographic Society in Lund, The
Swedish Royal Academy of Sciences and "LieGrits", a Marie Curie Research Training
Network funded by the European Community as project MRTN-CT 2003-505078.}}

\title{Simplicity and maximal commutative subalgebras of twisted generalized Weyl algebras}
\date{}

\maketitle

\begin{abstract}
In this paper we show that each non-zero ideal of a twisted generalized Weyl algebra (TGWA) $A$
intersects the centralizer of the distinguished subalgebra $R$ in $A$ non-trivially.
We also provide a necessary and sufficient condition for the centralizer of $R$ in $A$ to be commutative, and give examples of TGWAs associated to symmetric Cartan matrices
satisfying this condition.
By imposing a certain finiteness condition on $R$ (weaker than Noetherianity) we are able to make an Ore localization
which turns out to be useful when investigating simplicity of the TGWA.
Under this mild assumption we obtain necessary and sufficient conditions for the simplicity of TGWAs. We describe how this is
related to maximal commutativity of $R$ in $A$ and the (non-) existence of
non-trivial $\Z^n$-invariant ideals of $R$. Our result is a generalization of the rank one case, obtained by D. A. Jordan in 1993.
We illustrate our theorems by considering some special classes of TGWAs and providing concrete examples.
\end{abstract}

\tableofcontents

\pagestyle{headings}

\section{Introduction}

Higher rank generalized Weyl algebras (GWAs) were introduced by Bavula in \cite{Bavula} and are defined as follows.
Let $R$ be a ring and $\sigma : \Z^n\to\Aut(R)$, $g\mapsto\sigma_g$, an action of $\Z^n$ on $R$ by ring automorphisms, $\vect{t} = (t_1, \ldots , t_n)$ an $n$-tuple of non-zero elements of the center of $R$ such that $\sigma_i(t_j)=t_j$ for all $i\neq j$. The \emph{generalized Weyl algebra of rank} (or \emph{degree}) $n\in \Z_{>0}$, denoted
$R(\sigma,\vect{t})$, is the ring extension of $R$ by $X_1,\ldots,X_n,Y_1,\ldots,Y_n$ subject to the relations
\begin{subequations}
\begin{align}
X_ir&=\si_i(r)X_i,&Y_ir&=\si_i^{-1}(r)Y_i, &\forall r\in R, i\in \{1,\ldots,n\},\\
Y_iX_i&=t_i, &X_iY_i&=\si_i(t_i), &\forall i\in \{1,\ldots,n\},\\
Y_iY_j&=Y_jY_i, & X_iX_j&=X_jX_i  & \forall i,j \in \{1,\ldots,n\}, \label{eq:GWAcommutation}\\
&& X_iY_j&=Y_jX_i, & \forall i\neq j,
\end{align}
\end{subequations}
where $\si_i:=\si_{e_i}$ and $e_i=(0,\ldots,\overset{i}{1},\ldots,0)\in\Z^n$.
In \cite{MT_def} Mazorchuk and Turowska introduced a class of algebras called \emph{twisted generalized Weyl algebras} (TGWAs), which is a generalization of the higher rank generalized Weyl algebras (see Section \ref{TGWCTGWAdef} for the precise definition). Basically relation \eqref{eq:GWAcommutation} is relaxed in order to also accomodate situations where the defining relations among the $X_i$'s (and among the $Y_i$'s) are some $q$-commutation relations (say, $X_iX_j=qX_jX_i$, $i<j$) or some Serre-type relations (such as $X_i^2X_j-2X_iX_jX_i+X_jX_i^2=0$).

Already GWAs of rank one include many interesting algebras such as $U\big(\mathfrak{sl}(2)\big)$ and its quantization and deformations, as well as several quantized function spaces and so called generalized down-up algebras, see for example \cite{CasShe2004} and references therein.
This unification of different algebras into one family has proved to be very fruitful. For example in \cite{DGO} the authors classified all simple and indecomposable weight modules over a rank one GWA, which in particular gives in one stroke a description of such modules for characteristic zero, quantized, and modular $U\big(\mathfrak{sl}(2)\big)$.

One of the original motivations to introduce the TGWAs was that in higher rank several algebras which one would expect to be GWAs, such as the multiparameter quantized Weyl algebras defined in \cite{Malte}, are not GWAs (at least not in an obvious way). However they are TGWAs as was shown in \cite{MT02}.
In \cite{MPT} some algebras of importance in the representation theory of $\mathfrak{gl}(n)$, namely the Mickelsson-Zhelobenko algebra $Z\big(\mathfrak{gl}(n),\mathfrak{gl}(n-1)\big)$ and the extended orthogonal Gelfand-Tsetlin algebras (including certain localizations of $U\big(\mathfrak{gl}(n)\big)$), were shown to be examples of TGWAs. It is expected that the corresponding quantized versions of these algebras are also TGWAs. Sergeev \cite{Sergeev} proved that some primitive quotients of $U\big(\mathfrak{gl}(3)\big)$ are TGWAs and used this fact to produce multivariable analogs of Hahn polynomials.

When it comes to representation theory, several classes of simple weight (with respect to the commutative subring $R$) modules over TGWAs have been classified in \cite{Hartwig,MT_def,MPT}, including simple graded modules \cite{MPT}. 
Bounded and unbounded $\ast$-representations were described in \cite{MT02}.

In this paper we investigate TGWAs from another point of view, namely that of graded algebras; every TGWA of rank $n$ is $\Z^n$-graded in a natural way. It is known that the class of higher rank GWAs, hence the class of TGWAs, includes all skew group algebras over $\Z^n$ (take $t_i=1$ for all $i$ in the definition above). Another related fact, proved in \cite{FutHart}, is that any TGWA where $t_i$ is regular for each $i$ can be embedded in a crossed product algebra.

Consider the following problem:
\begin{prob}\label{prob:essential}
Given a $G$-graded algebra $A=\bigoplus_{g\in G}$ where $A_e$ is commutative,
is it true that each non-zero ideal of $A$ has non-zero intersection with the centralizer
of $A_e$ in $A$?
\end{prob}
This is known to be true for strongly group graded rings \cite{oin10} and for crystalline graded rings \cite{oin09}.
However, TGWAs are in general not strongly graded nor crystalline graded. 
The first main result of this paper (see Section \ref{IdealInterSectionsForTGWA}) shows that this also holds for TGWAs:
\begin{thmA}
Let $A=A\TGWdat$ be a twisted generalized Weyl algebra. Each non-zero ideal of $A$ has non-zero intersection with the centralizer of $R$ in $A$.
\end{thmA}
To establish this, we first prove that the gradation form (see Section \ref{sec:gradation_form}) of $A$ is always non-degenerate (Corollary \ref{cor:gamma_nondeg}), an interesting and useful fact in itself. Graded rings with this property have been studied before \cite{CohRow}. Then we can apply the general result from \cite[Theorem 3]{OL1}.
We introduce the notion of \emph{regularly graded algebras}, which is a generalization of crystalline graded algebras,
and provide exact conditions for a TGWA to be regularly graded (Theorem \ref{thm:regularly_graded_char}).

The second problem that we address is the following.
\begin{prob}
Given a $G$-graded algebra $A=\bigoplus_{g\in G} A_g$ where $A_e$ is commutative, when is
$A_e$ maximal commutative in $A$? When is the centralizer of $A_e$ in $A$ commutative?
\end{prob}
If one can prove that a commutative subalgebra $R$ has a commutative centralizer in $A$, then it
follows that the centralizer is the unique maximal commutative subalgebra in $A$
which contains $R$. This situation is of great value in representation theory when
studying weight modules over $A$ with respect to the commutative subalgebra $R$.
In \cite{MT_def} a sufficient condition for the degree zero subalgebra of a TGWA to have
a commutative centralizer was given.
We generalize this result by Theorem \ref{thm:commutant_suff} and Remark \ref{Rem:RDomZnSimple}
and provide necessary and sufficient conditions:
\begin{thmB}
Let $A=A\TGWdat$ be a twisted generalized Weyl algebra. Suppose that $R$ is a domain or that $R$ is $\Z^n$-simple with respect to the action $\sigma : \Z^n \to \Aut_{\K}(R)$. Then the centralizer of $R$ in $A$ is commutative if and only if $\ker(\si)$ has a $\Z$-basis $\{k_1,\ldots,k_s\}$ such that $[A_{mk_i},A_{lk_j}]=0$ for all $m,l\in\Z$ and all $i,j \in \{1,\ldots,s\}$, $i\neq j$, where $[a,b]=ab-ba$.
\end{thmB}
As an application we prove, in Theorem \ref{thm:TCcommutant}, that the family of TGWAs, $\mathcal{T}_q(C)$, parametrized by a scalar $q$ and a symmetric generalized Cartan matrix $C$, introduced in \cite{H09}, satisfies this condition (but not the condition in \cite{MT_def}). We also prove that the number $s$ in Theorem B in this case is equal to the number of connected components of the Coxeter graph of $C$ (Theorem \ref{thm:Kbas}).

Finally, the third problem we consider is the question of simplicity of TGWAs.
\begin{prob} Give necessary and sufficient conditions, expressed only in terms of the initial data $\TGWdat$, for a twisted generalized Weyl algebra $A\TGWdat$ to be simple.
\end{prob}
The phrase ``expressed in terms of the initial data'' is not meant to be precise, but rather to indicate the type of condition we aim for. Classical Weyl algebras are well-known to be simple, but generalized Weyl algebras are not always simple. A result of Jordan \cite[Theorem~6.1]{J93} provides a criterion for the simplicity of a degree one
generalized Weyl algebra  $R(\si,t)$, where $R$ is a commutative Noetherian ring,
 $\si\in\Aut(R)$ and $t\in R$. Namely,
 $R(\si, t)$ is simple if and only if (i) $\sigma$ has infinite
order; (ii) $t$ is a regular element in $R$; (iii) $R$ has no
$\sigma$-invariant ideals except $\{0\}$ and $R$; and (iv) for
all positive integers $m$, $R t + R \si^m(t) = R$.

One area where this type of results are of importance is the area of dynamical systems.
In \cite{SvenssonThesis} skew group algebras associated to dynamical systems are studied
and conditions for simplicity of certain $\Z$-graded skew group algebras and their analytical analogues are obtained.
The $\Z$-graded skew group algebras are in fact examples of GWAs of rank one.

In another paper by Jordan, \cite{J95}, the simplicity of a certain localization
of the multiparameter quantized Weyl algebra is proved. In \cite{MT02} it was proved
that the multiparameter quantized Weyl algebra is an example of a TGWA, and in \cite{FutHart}
that Jordan's simple localization is also a TGWA.

The third main result of this paper (Theorem \ref{thm:Simplicity_of_TGWAs}) unify and extend these two simplicity results by Jordan to
a large family of twisted generalized Weyl algebras.
\begin{thmC}
Let $A=A\TGWdat$ be a twisted generalized Weyl algebra where $R$ is Noetherian (or more generally, it is enough that $A$ is so called $R$-finitistic), and
where $t_i$ is regular in $R$ for each $i$. Then $A$ is simple if and only if
the following three assertions hold:
\begin{enumerate}[{\rm (i)}]
\item $AX_{i_1}X_{i_2}\cdots X_{i_m}A=A$ for all $m\in\Z_{\ge 0}$ and all $i_1,\ldots,i_m\in\{1,\ldots,n\}$;
\item $R$ is $\Z^n$-simple with respect to the action given by $\si$;
\item The center of $A$ is contained in $R$.
\end{enumerate}
\end{thmC}
The method we use to prove this result is a localization technique inspired by \cite{J95}. To be able to carry out the localization one needs to know that certain subsets of the algebra are Ore sets. It turns out that this 
condition is in fact equivalent to the TGWA to be so called \emph{$R$-finitistic}, a notion
introduced in \cite{H09} (where it was called ``locally finite over $R$'').
Both conditions hold if $R$ is Noetherian.
In a special case, still covering all higher rank GWAs, condition (i) of Theorem C can be made completely explicit (Theorem \ref{thm:Simplicity_of_TGWAs_of_type_A1^n}).

In Section \ref{Sect:Examples} we provide several interesting examples of TGWAs
and display some phenomena, which are not possible in other common classes of graded algebras.
The paper is concluded with a discussion about our results for TGWAs in the context of graded algebras.
We describe how our results for TGWAs are related to, and in several cases differ from,
the corresponding results for other classes of graded algebras.

\section*{Acknowledgements}
The first author is grateful to V. Futorny for interesting discussions.
The second author is grateful to P. Lundstr\"{o}m for interesting discussions.

\section{Preliminaries}
\subsection{Notation and definitions}\label{NotationDefinitions}
Rings and ring homomorphisms are always assumed to be unital. Ideals are understood
to be two-sided. Throughout this paper $\K$ will
denote an arbitrary commutative ring, unless otherwise stated.

Suppose that $S$ is a ring.
The group of units in $S$ is denoted by $U(S)$. 
An element $a\in S$ is said to be \emph{regular in $S$} if $a$ is
not a left nor a right zero-divisor in $S$.
The set of regular elements of $S$ is denoted by $S_{\reg}$.
If $a,b \in S$, then we let $[a,b]$ denote the element $ab-ba \in S$.
The \emph{centralizer} (or \emph{commutant}) of a subset $T\subseteq S$ is denoted by $C_S(T)$ and
is defined as the set $\big\{s\in S\mid st=ts,\;\,\forall t\in T\big\}$
which is clearly a subring of $S$.
If $T$ is commutative and $T=C_S(T)$, then $T$ is said to be \emph{maximal commutative} in $S$.
The \emph{center} of $S$, $C_S(S)$, is denoted by $Z(S)$.

If $X$ is a set, then the \emph{free $S$-bimodule on $X$} is defined as
$SXS:=\bigoplus_{x\in X}SxS$ where for $x\in X$ each summand $SxS$ is by definition isomorphic as an $S$-bimodule
to $S\otimes_\Z S$ via $s_1 x s_2\mapsto s_1\otimes s_2$.
The \emph{free $S$-ring $F_S(X)$ on $X$} is defined as the tensor algebra 
of $SXS$ over $S$. Namely, $F_S(X):=\bigoplus_{n=0}^\infty (SXS)^{\otimes_S n}$,
where $(SXS)^{\otimes_S n}=(SXS)\otimes_S \cdots \otimes_S(SXS)$ ($n$ factors)
and $(SXS)^{\otimes_S 0}=S$ by convention. Thus, a general element in $F_S(X)$
is a sum of monomials of the form $s_1 x_1 s_2 x_2\cdots s_k x_k s_{k+1}$
where $s_i\in S$ and $x_i\in X$, for $i\in \{1,\ldots,k\}$. There is a natural ring homomorphism
$S\to F_S(X)$ given by inclusion into the degree zero component.

Suppose that $S$ is a $\K$-algebra, i.e.
$S$ is a ring together with a ring homomorphism $\eta_S:\K\to S$
such that $\eta_S(\la)s=s\eta_S(\la)$ for all $s\in S$, $\la\in \K$.
Suppose that $S'$ is another $\K$-algebra. A $\K$-algebra homomorphism $\varphi:S\to S'$
is a ring homomorphism 
such that $\varphi\circ \eta_{S}=\eta_{S'}$.
If $X$ and $Y$ are nonempty subsets of $S$,
then $XY$ denotes the $\K$-linear span of the set $\{xy\mid x\in X, y\in Y\}$.

If $G$ is a group acting as
$\K$-algebra automorphisms of $S$, $\rho: G \ni g\mapsto \rho_g\in\Aut_\K(S)$,
 then an ideal $I \subseteq S$ is said to be \emph{$G$-invariant} if
$\rho_g(I) = I$ for all $g\in G$. If $S$ and $\{0\}$ are the only $G$-invariant ideals of $S$, then $S$ is said to be \emph{$G$-simple}.

Recall that, given a group $G$, a \emph{$G$-gradation on $S$} is a
 set of $\K$-submodules, $\{S_g\}_{g \in G}$, of $S$ such that $S =
\bigoplus_{g \in G} S_g$ and $S_g S_h \subseteq S_{gh}$, for $g,h \in G$.
If in addition $S_g S_h = S_{gh}$, for all $g,h \in G$, then the gradation is a \emph{strong $G$-gradation}.
A \emph{(strongly) $G$-graded $\K$-algebra} is a $\K$-algebra together with a (strong) $G$-gradation on it.
Each element $a$ of a $G$-graded $\K$-algebra $S$ can be written as $a= \sum_{g \in G} a_g$
where $a_g \in S_g$, for $g\in G$, and $a_g=0$ for all but finitely many $g\in G$.
The \emph{support} of $a$ is defined as the set $\Supp(a)=\{g\in G \,\mid\, a_g \neq 0\}$,
the cardinality of which is denoted by $|\Supp(a)|$.
For $g\in G$, the $\K$-submodule $S_g$ is referred to as the \emph{homogeneous component}
of degree $g$. An element $s\in S$ is said to be \emph{homogeneous of degree $g$},
written $\deg(s)=g$, if $s\in S_g$ for $g\in G$.
The neutral element of $G$ is denoted by $e$ and the subring $S_e$
 is refered to as the \emph{neutral component} of $S$.
A left (right, two-sided) ideal of a $G$-graded $\K$-algebra $S$ is said to be \emph{graded} if 
$I = \sum_{g\in G} I_g$, where $I_g:= S_g\cap I$ for $g\in G$.
If $I$ is a graded ideal, then the quotient ring $S/I$ is a $G$-graded $\K$-algebra $\bigoplus_{g\in G} (S/I)_g$ with a gradation defined by $(S/I)_g = \big\{s+I\mid s\in S_g\big\}$ for $g\in G$.

\subsection{Bilinear forms and radicals associated to group graded algebras}
\subsubsection{The gradation form}\label{sec:gradation_form}
Let $G$ be a group with neutral element $e$ and $A=\bigoplus_{g\in G} A_g$ a $G$-graded $\K$-algebra.
(By this it is understood that, for each $g\in G$, $A_g$ is a $\K$-submodule of $A$.)
Denote by $\mathfrak{p}_e:A\to A_e$ the graded projection onto the neutral component,
i.e $\mathfrak{p}_e(\sum_{g\in G} a_g)=a_e$, where $a_g\in A_g$ for $g\in G$.
To each $G$-graded $\K$-algebra $A$ we associate a map, which we shall refer to as the \emph{gradation form} of $A$.
It is defined in the following way:
\begin{equation}
\begin{gathered}
\gamma_A :A\times A\to A_e, \quad (a,b) \mapsto \mathfrak{p}_e(ab) \\
\end{gathered}
\end{equation}
This form was introduced by Cohen and Rowen in \cite{CohRow} in their study of $G$-graded rings.
Note that $\gamma_A$ depends on the gradation on $A$.
It is immediate that $\ga_A$ is $\K$-bilinear and that
\begin{equation}\label{eq:gamma_adjoint}\ga_A(a,bc)=\ga_A(ab,c)
\end{equation}
for all $a,b,c\in A$. Furthermore, if $a=\sum_{g\in G} a_g$ and $b=\sum_{g\in G}b_g$ where $a_g,b_g\in A_g$, for $g\in G$, then
\begin{equation}\label{eq:gamma_formula} \ga_A(a,b)=\sum_{g\in G}a_g b_{g^{-1}}. \end{equation}
The left respectively right radical of $\ga_A$ is defined as follows:
\begin{align}
\Rad_\mathrm{l}(\gamma_A):=\big\{b\in A\mid \gamma_A(a,b)=0, \,\,\forall a\in A\big\},\\
\Rad_\mathrm{r}(\gamma_A):=\big\{a\in A\mid \gamma_A(a,b)=0, \,\,\forall b\in A\big\}.
\end{align}
It is easy to check that
$\Rad_\mathrm{l}(\gamma_A)$ is a graded left ideal of $A$,
while $\Rad_\mathrm{r}(\gamma_A)$ is a graded right ideal of $A$.

\begin{definition}\label{FormProperties}
Let $A$ be a $G$-graded $\K$-algebra with gradation form $\ga_A$.
\begin{enumerate}[(i)]
	\item If $\Rad_\mathrm{l}(\ga_A)=\Rad_\mathrm{r}(\ga_A)$, then $\ga_A$ is said to be \emph{radical-symmetric}
	and the left (and right) radical is simply denoted by $\Rad(\ga_A)$.
	\item We say that $\ga_A$ is \emph{non-degenerate} if it is radical-symmetric and $\Rad(\ga_A)=\{0\}$.
	\item If there exists an action of $G$ by $\K$-algebra automorphisms on $A_e$ (denoted $G\ni g\mapsto \rho_g\in\Aut_\K(A_e)$)
 such that, for any $g\in G$,
\begin{equation}
\ga_A(a,b)=\rho_g\big(\ga_A(b,a)\big)
\end{equation}
for all $a\in A_g$ and all $b\in A_{g^{-1}}$, then $\ga_A$ is said to be \emph{$G$-symmetric}.
\end{enumerate}
\end{definition}

\begin{remark}
Note that if $\ga_A$ is $G$-symmetric, then it is radical-symmetric.
Moreover, if $\ga_A$ is radical-symmetric, then $\Rad(\ga_A)$ is a two-sided graded ideal of $A$.
\end{remark}

\begin{remark}\label{RemarkNonDegCohRowOinertLundstrom}
Also note that $\gamma_A$ is non-degenerate in the sense of Definition \ref{FormProperties}
if and only if
it is non-degenerate in the sense of Cohen and Rowen \cite{CohRow}.
Furthermore, if this is the case, then the gradation on $A$ is left (and right)
non-degenerate in the sense of \cite[Definition 2]{OL1}.
\end{remark}

If $A$ is a $G$-graded $\K$-algebra with radical-symmetric gradation form $\ga_A$,
then we put $\bar A:=A/\Rad(\ga_A)$. In that case $\Rad(\ga_A)$ is a graded ideal
and hence the quotient algebra $\bar A$ has a natural $G$-gradation induced by
the one on $A$; 
$\bar A=\bigoplus_{g\in G} \bar A_g$, where
 $\bar A_g=\big\{a+\Rad(\ga_A)\in\bar A\mid a\in A_g\big\}$ for $g\in G$.
 For each $g\in G$ we have that $\bar A_g\simeq A_g/(A_g\cap\Rad(\ga_A))$ as $\K$-modules.
 Since $\bar A$ is a $G$-graded $\K$-algebra, this allows us to define a gradation form on
 this algebra as well.
 The following lemma is proved under the hypothesis that $\gamma_A$ is radical-symmetric.
 
\begin{lemma}
The gradation form $\ga_{\bar A}$ on $\bar A$ is non-degenerate.
\end{lemma}
\begin{proof}
Since $\ga_A$ is radical-symmetric, so is $\ga_{\bar A}$.
Take an arbitrary $\bar a\in\Rad(\ga_{\bar A})$, i.e. such that
 $\ga_{\bar A}(\bar a,\bar b)=0$ for all $\bar b\in \bar A$.
Choose $a\in A$ such that $\bar a=a+\Rad(\ga_A)$. 
By \eqref{eq:gamma_formula} we get
 $\ga_A(a,b)+\Rad(\ga_A)=\sum_{g\in G} a_gb_{g^{-1}}+\Rad(\ga_A)=\sum_{g\in G} \bar a_g\bar b_{g^{-1}}=\ga_{\bar A}(\bar a,\bar b)=0$ in $\bar A$, for any $b\in A$.
Thus $\ga_A(a,b)\in \Rad(\ga_A)\cap A_e$ for all $b\in A$. However, $1\in A_e$ and $\ga_A(r,1)=r=\ga_A(1,r)$ for each $r\in A_e$ and hence
$\Rad(\ga_A)\cap A_e=\{0\}$. This shows that $\ga_A(a,b)=0$ for all $b\in A$, which means that $a\in \Rad(\ga_A)$. Thus $\bar a=0$ in $\bar A$, proving that
$\Rad(\gamma_{\bar A})=\{0\}$, i.e. $\gamma_{\bar A}$ is non-degenerate.
\end{proof}

\subsubsection{The gradical}
Let $A$ be a $G$-graded $\K$-algebra.
The \emph{gradical} of $A$ (short for \emph{gradation radical}),
denoted by $\grRad(A)$, is defined as the sum of all graded ideals $J$ of $A$ which satisfy
$A_e\cap J=\{0\}$. The gradical $\grRad(A)$ is itself a graded ideal of $A$ such that $A_e\cap \grRad(A)=\{0\}$,
and by construction it is the unique maximal element in the set of all such graded ideals (ordered by inclusion).

The gradical $\grRad(A)$ is related to the radicals of the gradation form $\ga_A$ of $A$,
defined in Section \ref{sec:gradation_form}, in the following way:

\begin{prop}\label{prop:gradation_things}
Let $A$ be a $G$-graded $\K$-algebra with gradation form $\ga_A$. The following holds:
\begin{enumerate}[{\rm (i)}]
\item $\grRad(A)\subseteq \Rad_\mathrm{l}(\ga_A)\cap \Rad_\mathrm{r}(\ga_A)$;
\item If $\ga_A$ is $G$-symmetric, then $\Rad(\ga_A)=\grRad(A)$.
\end{enumerate}
\end{prop}
\begin{proof}
(i) Let $a\in \grRad(A)$. Since $\grRad(A)$ is graded we can assume that $a$ is homogeneous,
i.e. $a\in A_g\cap \grRad(A)$ for some $g\in G$. Let $b\in A$ be arbitrary.
Write $b=\sum_{h\in G} b_h$ where $b_h\in A_h$ for $h\in G$.
Then, by \eqref{eq:gamma_formula}, $\ga_A(a,b)= ab_{-g} \in A_e\cap \grRad(A) =\{0\}$.
Also $\ga_A(b,a)=b_{-g}a\in A_e\cap \grRad(A)=\{0\}$.
Since $b$ was arbitrary, this proves that $a\in \Rad_\mathrm{l}(\ga_A)\cap \Rad_\mathrm{r}(\ga_A)$.

(ii) If $\ga_A$ is $G$-symmetric, then it is radical-symmetric, so by part (i)
it only remains to prove that $\Rad(\ga_A)\subseteq \grRad(A)$. Let $a\in\Rad(\ga_A)$.
Since $\Rad(\ga_A)$ is graded we can assume that $a\in A_g$ for some $g\in G$.
Put $J=AaA$, the ideal in $A$ generated by $a$. Since $a$ is homogeneous,
$J$ is graded. If we show that $A_e\cap J=\{0\}$ then it follows that
$J\subseteq \grRad(A)$, and in particular $a\in \grRad(A)$.
Any element $x$ of $A_e\cap J$ can be written as
\[x=\sum_{\substack{h,k\in G\\hgk=e}} b_hac_k\]
where $b_h,c_h\in A_h$ for all $h\in G$. But then, for any $h,k\in G$
with $hgk=e$ we have,
using \eqref{eq:gamma_adjoint} and that $\ga_A$ is $G$-symmetric,
\[b_hac_k= \ga_A(b_h,ac_k) = \rho_h\big(\ga_A(ac_k,b_h)\big)=\rho_h\big(\ga_A(a,c_kb_h)\big)=0\]
since $a\in\Rad(\ga_A)$, proving that $x=0$.
Since $x$ was arbitrary, $A_e\cap J=\{0\}$.
As shown above, this implies that $a\in \grRad(A)$. But $a$ was arbitrary
so we conclude that $\Rad(\ga_A)\subseteq \grRad(A)$.
\end{proof}

\subsection{The TGWC and TGWA}\label{TGWCTGWAdef}
We shall now recall the definition of a twisted generalized Weyl algebra from \cite{MT_def}, \cite{MT02}.
Fix a positive integer $n$ and set $\nn=\{1,2,\ldots,n\}$.
Let $\K$ be a commutative unital ring,
and let $R$ be a commutative unital $\K$-algebra,
$\si:\Z^n\to\Aut_\K(R)$, $\si:g\mapsto\si_g$, be a group homomorphism
from $\Z^n$ to the group of $\K$-algebra automorphisms of $R$,
$\mu=(\mu_{ij})_{i,j\in\nn}$ be a matrix with
invertible entries from $\K$
and $\vect{t}=(t_1,\ldots,t_n)$ be an $n$-tuple of central elements of $R$.
The quadruple $\TGWdat$ will be referred to as a \emph{TGW datum}.
For convenience we will denote $\si_{e_i}$, where $e_i=(0,\ldots,\overset{i}{1},\ldots,0)\in\Z^n$ by $\si_i$.
Conversely, given $n$ commuting $\K$-algebra automorphisms $\sigma_1,\ldots,\sigma_n$ of $R$ we put
$\si_g=\si_1^{g_1}\circ\cdots\circ\si_n^{g_n}$ for all $g\in\Z^n$ and this defines a group homomorphism $\sigma : \Z^n \to \Aut_{\K}(R)$.

The \emph{twisted generalized Weyl construction} (TGWC) $A'=A'\TGWdat$ obtained from
the TGW datum $\TGWdat$
is the free $R$-ring $F_R(Z)$ on a set $Z=\{X_i, Y_i \mid i\in\nn\}$ of $2n$ symbols
modulo the following relations:
\begin{subequations}\label{eq:rels}
\begin{align}
X_ir&=\si_i(r)X_i,&Y_ir&=\si_i^{-1}(r)Y_i, &\text{for }r\in R, i\in\nn,
\label{eq:rel1}\\
Y_iX_i&=t_i, &X_iY_i&=\si_i(t_i), &\text{for }i\in\nn, \label{eq:rel2}\\
X_iY_j&=\mu_{ij}Y_jX_i, &&&\text{for }i,j\in\nn, i\neq j.
\label{eq:rel3}
\end{align}
\end{subequations}
Due to relation \eqref{eq:rel2}, it is not guaranteed that the
 $\K$-algebra homomorphism
$\iota:R\to A'$ given by composing the natural map $R\to F_R(Z)$
with the canonical projection $F_R(Z)\to A'$ is injective.
However, \emph{throughout this paper we will make the additional
assumption that $\iota$ is injective}. 
In \cite[Corollary 2.17]{FutHart} it was proved that if
$t_1,\ldots,t_n$ are regular elements of $R$, then $\iota$ is injective
if and only if the following consistency conditions hold:
\begin{subequations}\label{eq:tgwa_consistency}
\begin{alignat}{2}
\label{eq:tgwa_consistency1}
\si_i\si_j(t_it_j)&=\mu_{ij}\mu_{ji}\si_i(t_i)\si_j(t_j),& \qquad\forall i,j\in\nn,\, i\neq j,\\
\label{eq:tgwa_consistency2}
t_j\si_i\si_k(t_j)&=\si_i(t_j)\si_k(t_j), &\qquad\forall i,j,k\in\nn,\, i\neq j\neq k\neq i.
\end{alignat}
\end{subequations}
Condition \eqref{eq:tgwa_consistency1} appeared already in \cite{MT_def,MT02}.
Thus, when considering examples where the $t_i$'s are regular in $R$, it is enough
to verify relations \eqref{eq:tgwa_consistency} to know that $\iota:R\to A'$ is injective.

Relations (\ref{eq:rels}) are homogeneous with respect to the $\Z^n$-gradation on
the free $R$-ring $F_R(Z)$ uniquely defined by requiring
\[\deg X_i=e_i, \quad\deg Y_i=-e_i,\quad \deg r=0,\quad\text{for }i\in\nn,
r\in R,\]
where $e_i=(\overset{1}{0},\ldots,0,\overset{i}{1},0,\ldots,\overset{n}{0})$.
Thus this $\Z^n$-gradation on $F_R(Z)$ descends to a $\Z^n$-gradation
$\{A'_g\}_{g\in\Z^n}$ on the quotient $A'$.
It is easy to see that $A'_0=\iota(R)$. Thus, since we always assume that
$\iota$ is injective, we may identify $R$ with $A'_0$ via the isomorphism $\iota$.

The \emph{twisted generalized Weyl algebra} (TGWA)
$A=A\TGWdat$ of \emph{rank} $n$ is defined to be $A'/I$, where $I=\grRad(A')$
is the gradical of $A'$ (i.e. the sum of all graded two-sided ideals of $A'$ intersecting $R$
trivially). Since $I$ is graded, $A$ inherits a $\Z^n$-gradation
$\{A_g\}_{g\in\Z^n}$ from $A'$.
Note that $R$ is isomorphic to its image in the quotient $A=A'/I$. We shall therefore identify
these algebras and denote them with the same letter $R$.

By a \emph{monic monomial} in a TGWC $A'$ or TGWA $A$
we mean a product of the generators $X_i$, $Y_i$, for $i\in \nn$.
A product of an element of $R$ and a monic monimial is referred to as a \emph{monomial}.
A monic monomial is called \emph{reduced} if it has the form
\begin{equation}
Y_{i_1}\cdots Y_{i_k} X_{j_1}\cdots X_{j_l},\quad
 \text{where $i_r,j_s\in\nn$ and $\{i_1,\ldots,i_k\}\cap \{j_1,\ldots,j_l\}=\emptyset$.}
\end{equation}
The following lemma is useful and straightforward to prove.
\begin{lemma}\label{LEMMATWOSIX}
$A'$ (respectively $A$) is generated as a left and as a right $R$-module by the reduced monomials in $A'$
 (respectively $A$).
\end{lemma}

\section{Essential subalgebras of TGWAs}\label{IdealInterSectionsForTGWA}

The \emph{ideal intersection property} has been studied in e.g. \cite{OL1}.
In this paper, subalgebras having this property are said to be \emph{essential}
and they are defined as follows.

\begin{definition}\label{def:essential_subalgebra}
A subalgebra $S'$ of a $\K$-algebra $S$ is said to be an \emph{essential subalgebra of $S$}, if $S' \cap I \neq \{0\}$ for each non-zero ideal $I$ of $S$.
\end{definition}

One useful fact, is that injectivity of a $\K$-algebra homomorphism $\varphi : S \to T$
follows from the injectivity of $\varphi\lvert_{S'} : S' \to T$
(the restriction of $\varphi$ to $S'$) if $S'$ is an essential subalgebra of $S$.

Let $A$ be a group graded $\K$-algebra with commutative neutral component $A_e$.
Consider the following question:
Is the centralizer of $A_e$ in $A$ an essential subalgebra of $A$? The answer is known to be affirmative
if $A$ is strongly graded \cite{oin10} or if $A$ is a \emph{crystalline graded ring} (CGR) \cite{oin09},
or more generally for any ring satisfying the condition in \cite[Theorem 3]{OL1}.
We shall now prove that the answer is affirmative also for TGWAs. This is interesting
since there are TGWAs which, with their natural $\Z^n$-gradation, are neither CGRs nor strongly graded.
We begin with some useful preliminary resuts.

The following interesting and useful commutation relation (proved under
additional assumptions in \cite{FutHart}, \cite{H09})
shows that the gradation form $\ga_{A'}$ on an arbitrary TGWC $A'$ is always $\Z^n$-symmetric (in the sense of Definition \ref{FormProperties}).
\begin{lemma}\label{lem:nondegtgwaflip}
Let $A'=A'\TGWdat$ be a TGWC. For any $g\in \Z^n$, we have
\begin{equation}
ab=\si_g(ba)
\end{equation}
for all $a\in A'_g$ and all $b\in A'_{-g}$. Hence the gradation form $\ga_{A'}$ on $A'$ is $\Z^n$-symmetric.
\end{lemma}
\begin{proof}
Let $g\in \Z^n$ be arbitrary. It suffices to prove the statement when $a \in A'_g$ is a monomial.
Also, we may assume that $a$ is monic because if the statement holds for a monic monomial then $rab=r\si_g(ba)=\si_g(\si_{-g}(r)ba)=\si_g(bra)$ for any $r\in R$, any $b\in A'_{-g}$ and monic monomial $a$.
If we first suppose that $a=X_i$, then $b$ has degree $-e_i$ and hence $b=rY_i$ for some $r\in R$,
by Lemma \ref{LEMMATWOSIX}.
We then get $ab=X_i r Y_i =\si_i(r)\si_i(t_i) = \si_i(rY_iX_i)=\si_i(ba)$.
The case $a=Y_i$ is treated analogously.
By iterating this argument we find that the statement holds for an arbitrary monic monomial $a$.
\end{proof}

By Proposition \ref{prop:gradation_things} we obtain the following important result.
\begin{corollary}\label{cor:gamma_nondeg}
The gradical $I=\grRad(A')$ of a TGWC $A'$ is equal to the radical of the gradation form $\ga_{A'}$ on $A'$,
and the gradation form $\ga_A$ on the corresponding TGWA $A=A'/\grRad(A')$ is non-degenerate.
\end{corollary}

\begin{remark}\label{RemarkNonDegCohRowOinertLundstromForTGWAs}
By Remark \ref{RemarkNonDegCohRowOinertLundstrom} this shows that
for a general TGWA $A$, $\gamma_A$ is non-degenerate in the sense of Cohen and Rowen \cite{CohRow}
and furthermore, the grading on $A$ is left (and right)
non-degenerate in the sense of \cite[Definition 2]{OL1}.
\end{remark}

\begin{remark}
Let $A'=A'\TGWdat$ be a TGWC. If we assume that $\mu$ is symmetric, then we can define
an involution $* : A' \to A'$.
Recall the left and right Shapovalov-type forms from \cite{MPT},
 $F^\mathrm{l}:A'\times A'\to R$ and $F^\mathrm{r}:A'\times A'\to R$
 defined by $F^\mathrm{l}(a,b)=\mathfrak{p}_0(a^\ast b)$
 and $F^\mathrm{r}(a,b)=\mathfrak{p}_0(ab^\ast)$ respectively.
In \cite{MPT} it was shown that, if $R$ is a domain, then the
radicals of $F^\mathrm{l}$ and $F^\mathrm{r}$, i.e. 
$\Rad(F^\mathrm{r}) := \{a\in A' \,\mid \, F^r(a,b)=0 \text{ for any } b\in A\}$
respectively 
$\Rad(F^\mathrm{l}) := \{a\in A' \,\mid \, F^l(b,a)=0 \text{ for any } b\in A\}$, coincide with eachother
and with the ideal $I$.
Corollary \ref{cor:gamma_nondeg} generalizes
this relation to the case of arbitrary TGWAs and arbitrary $\mu$.
\end{remark}

\begin{theorem}\label{IdealIntersection}
Let $A=A\TGWdat$ be a TGWA. Then $C_A(R)$ is an essential subalgebra of $A$.
That is, each non-zero ideal of $A$
intersects the centralizer of $R$ in $A$ non-trivially.
\end{theorem}
\begin{proof}
This follows from Corollary \ref{cor:gamma_nondeg} (Remark \ref{RemarkNonDegCohRowOinertLundstromForTGWAs})
 and \cite[Theorem 3]{OL1}. For the convenience of the reader
we give a proof here.
Let $J$ be a non-zero ideal of $A$. Among all non-zero elements of $J$,
choose one, $a=\sum_{g\in\Z^n} a_g$, where $a_g\in A_g$ for $g\in G$,
such that its support is minimal, i.e. such that $|\Supp(a)|$
is as small as possible.
Pick some $h\in\Z^n$ such
that $a_h\neq 0$. Then there are homogeneous elements $b,c\in A$ such that
$\deg(b)+\deg(c)=-h$ and $ba_hc\neq 0$. Indeed, otherwise $a_h$
would generate a graded ideal in $A$ with zero intersection with $R$.
Put $a'=bac$. Then $a'$ is a non-zero element of $J$
such that $|\Supp(a')|=|\Supp(a)|$, but in addition $a'_0 = ba_hc\neq 0$.
We claim that $a'\in C_A(R)$. Otherwise there is an $r\in R$
with $a''=[a',r]\neq 0$. But then $a''$ is a non-zero element of
$J$
such that $|\Supp(a'')| < |\Supp(a)|$ because
$a''_0=[a'_0,r]=0$ since $a'_0\in R$ and $a'_0\neq 0$.
This contradicts the minimality
of $|\Supp(a)|$. Hence $a'\in C_A(R)\cap J$.
This proves that $C_A(R)\cap J\neq \{0\}$.
\end{proof}

\begin{remark}
Note that the statement in the above proof is concerned with \emph{ideals}, not only graded ideals.
It is immediate from the definition that 
any non-zero graded ideal of a TGWA intersects $R$ non-trivially.
\end{remark}

The following example shows two things:
\begin{enumerate}[{\rm (i)}]
	\item The corresponding statement is not true for all GWAs.
	\item Not all GWAs are isomorphic to a CGR with neutral component $R$.
\end{enumerate}

\begin{exmp}
Put $R=\C[H]/(H^2)$, and let $t=H+(H)$. Let $\varepsilon\in \C$ be a non-root of unity
and define $\si\in\Aut_\C(R)$ by $\si(t)=\varepsilon t$. Consider $A=R(\si,t)$, the
corresponding generalized Weyl algebra of rank one (equivalently, $A$
is a TGWC of rank one). We claim that $C_A(R)=R$. Suppose that $a\in C_A(R)$.
Write $a=\sum_{k\in\Z} r_k Z^{(k)}$ where $Z^{(k)}$ equals $X^k$ if $k\ge 0$
and $Y^{-k}$ otherwise.
Then $0=[a,t]=\sum_{k\in\Z} r_k(\varepsilon^k-1)Z^{(k)}$. Since $A$ is $\Z$-graded
we get $r_k(\varepsilon^k-1)Z^{(k)}=0$ for each $k\in\Z$. It is well-known that
each $Z^{(k)}$ generates a free left $R$-module. Thus, since $\varepsilon$ is
not a root of unity, $r_k=0$ for all $k\in\Z\setminus\{0\}$. Hence $C_A(R)=R$.
However, it is easy to see that $X^2$ generates an ideal in $A$ which has
zero intersection with $R$.
\end{exmp}

\section{Characterization of regularly graded TGWAs}

Crystalline graded algebras were introduced by Nauwelaerts and Van Oystaeyen in \cite{CGR}
as a natural generalization of e.g. GWAs and $G$-crossed products.
Not all TGWAs fit into this class. We shall now introduce
the notion of \emph{regularly graded algebras}, of which crystalline graded algebras are a special case.

\begin{definition}
A $G$-graded $\K$-algebra is said to be \emph{regularly graded} if each homogeneous
component contains a regular element. In this case the gradation is said to be \emph{regular}.
\end{definition}

Note that the gradation form of a regularly graded algebra is necessarily non-degenerate.
Moreover, a regular gradation is necessarily \emph{faithful} (called \emph{component regular} in \cite{Passman}) in the sense of Cohen and Montgomery \cite{CohMont}.

\begin{remark}
As already mentioned, $G$-crossed products are examples of crystalline graded algebras, and hence they are regularly graded.
Moreover, each $G$-crossed product is strongly $G$-graded. However, not every strongly graded algebra is
regularly graded. To see this, let $A:=M_3(\C)$ be the algebra of $3 \times 3$-matrices over $\C$. It is possible to
define a strong $\Z_2$-gradation on $A=A_{\bar{0}} \oplus A_{\bar{1}}$ (see e.g. \cite[Example 6.11]{Oinert}) and one can show that $A_{\bar{1}}$
contains no regular element. Therefore $A$ is strongly $\Z_2$-graded (hence faithfully graded and $\gamma_A$ is non-degenerate) but not regularly graded.
\end{remark}

The following theorem gives a characterization of those TGWAs whose $\Z^n$-gradation
is regular.

\begin{theorem}\label{thm:regularly_graded_char}
Let $A=A\TGWdat$ be a TGWA. The following assertions are equivalent:
\begin{enumerate}[{\rm (i)}]
	\item \label{it:regularly_graded} $A$ is regularly graded;
  \item \label{it:ti_regular} For each $i \in \nn$, $t_i \in R_{\reg}$;
	\item \label{it:torsionfree} Each monic monomial in $A$ is non-zero
    and generates a free left (and right) $R$-module of rank one;
  \item \label{it:regular_property} If $a\in A$ is a homogeneous element such that $bac=0$
       for some monic monomials $b,c\in A$ with
       $\deg(a)+\deg(b)+\deg(c)=0$, then $a=0$.
\end{enumerate}
\end{theorem}
\begin{proof}
\eqref{it:regularly_graded} $\Rightarrow$ \eqref{it:ti_regular}:
If $A$ is regularly graded, then for each $i\in\nn$, $A_{e_i}$ contains
a regular element. Since $A_{e_i}=X_iR$ it means there exists an $r\in R$
such that $X_ir=\si_i(r)X_i$ is regular in $A$. But then $X_i$ must be regular in $A$
as well. Analogously $Y_i$ must be regular in $A$. Thus $Y_iX_i=t_i$ is regular
in $R$.

\eqref{it:ti_regular} $\Rightarrow$ \eqref{it:regularly_graded}:
Assume that $t_i$ is regular in $R$ for all $i\in\nn$.
Let $g\in\Z^n$ be arbitrary and let $a=Z_1^{(g_1)}\cdots Z_n^{(g_n)}$,
where $Z_i^{(k)}=X_i^k$ if $k\ge 0$ and $Z_i^{(-k)}=Y_i^{-k}$ if $k<0$.
Put $a^*=Z_n^{(-g_n)}Z_{n-1}^{(-g_{n-1})}\cdots Z_1^{(-g_1)}$.
Then $a a^*$ and $a^*a$ are products of elements of the form $\si_h(t_i)$
($h\in \Z^n, i\in\nn$), hence they are both regular in $R$.
If $b\in A$ with $ab=0$, then $a^*ab=0$ so $a^*ab_h=0$ for each
homogeneous component $b_h$ of $b$. But if $b_h\neq 0$ for some $h\in \Z^n$
then by Corollary \ref{cor:gamma_nondeg} there is some $c\in A_{-h}$
such that $b_hc\neq 0$. Then $a^*ab_hc=0$ and $b_hc\in R\setminus\{0\}$
contradict the fact that $a^*a$ is regular in $R$. So $b=0$. Similarly $ba=0$
for some $b\in A$ implies $b=0$. Thus $a$ is regular in $A$.

\eqref{it:ti_regular} $\Rightarrow$ \eqref{it:torsionfree}:
Let $a \in A$ be a monic
monomial and let $r\in R\setminus\{0\}$.
As in the previous step, $aa^*$ is a regular element of $R$, hence $raa^* \neq 0$ and thus $ra\neq 0$.

\eqref{it:torsionfree} $\Rightarrow$ \eqref{it:regular_property}:
By Lemma \ref{lem:nondegtgwaflip}, we have $bac=\si_h(acb)$ where
$h=\deg(b)$. Thus $acb=\si_h^{-1}(bac)=0$. 
If $a\neq 0$, then by Corollary \ref{cor:gamma_nondeg}, $da\in R\setminus\{0\}$
for some $d\in A$ with $\deg(d)=-\deg(a)$. But by
\eqref{it:torsionfree}, $da\cdot cb=0$ implies that $da=0$.
This contradiction shows that $a=0$.

\eqref{it:regular_property} $\Rightarrow$ \eqref{it:ti_regular}:
Suppose that $rt_i=0$ for some $r\in R, i\in\nn$. Then $0=rt_i=rY_iX_i=Y_i\si_i(r)X_i$.
By \eqref{it:regular_property} we get $\si_i(r)=0$ and hence $r=0$. This
proves that $t_i$ is regular in $R$ for each $i\in\nn$.
\end{proof}

\begin{remark}
Property \eqref{it:regular_property} in the above theorem is very
convenient for proving relations in a regularly graded TGWA.
\end{remark}

\section{Centralizers of $R$ in TGWAs}
\subsection{Results for general TGWAs}
Given a TGW datum $\TGWdat$, we shall denote the kernel of the group homomorphism $\si:\Z^n\to\Aut_\K(R)$ by $K$.
Also we use the notation 
$Z_i^{(k)}=\begin{cases} X_i^k,&k\ge 0,\\ Y_i^{-k},&k<0\end{cases}$
where $i\in \nn$ and $k\in\Z$.

\begin{theorem}\label{CommutantDescription}
Let $A=A\TGWdat$ be a TGWA and let $A_K$ be the following subalgebra of $A$:
\[A_K:=\bigoplus_{g\in K} A_g.\]
Then $A_K$ is contained in $C_A(R)$, the centralizer of $R$ in $A$.
Moreover, if $R$ is a domain or if $R$ is $\Z^n$-simple, then $A_K=C_A(R)$.
\end{theorem}
\begin{proof}
Take an arbitrary $g=(g_1,\ldots,g_n)\in K$ and $a\in A_g$.
Suppose that $a\notin C_A(R)$. Then there exists
an $r\in R$ such that $ar\neq ra$. On the other hand, $ar=\si_g(r)a$
and hence $\si_g(r)$ cannot be equal to $r$. This shows that $\si_g\neq\identity_R$,
which contradicts $g\in K$. Therefore $a\in C_A(R)$. Since $a$ and $g$ were
arbitrary, this shows that $A_K \subseteq C_A(R)$.

For the converse, take an arbitrary non-zero $a\in C_A(R)$ (clearly $0 \in A_K \cap C_A(R)$)
and let $r\in R$ be arbitrary.
Write $a=\sum_{g\in\Z^n} a_g$, with $a_g\in A_g$ for $g\in \Z^n$. Then
$0=[a,r]=\sum_{g\in\Z^n} [a_g,r]$.
Since $\deg r=0$ we have $[a_g,r]\in A_g$ for each $g\in G$
and from the gradation on $A=\bigoplus_{g\in\Z^n} A_g$ we conclude that $[a_g,r]=0$ for each $g\in G$.
Pick an arbitrary $g\in G$ for which $a_g\neq 0$. Then, since $r$ was arbitrary,
\begin{equation}\label{eq:maxcomm_pf}
0=[a_g,r]=(\si_g(r)-r)a_g , \quad\forall r\in R.
\end{equation}
Since $a_g\neq 0$ there is, by Corollary \ref{cor:gamma_nondeg},
an element $c\in A_{-g}$ such that $a_gc\neq 0$.
Multiplying \eqref{eq:maxcomm_pf} from the right by $c$ we get
$(\si_g(r)-r)a_gc=0$ for all $r\in R$, and $0\neq a_gc\in R$. If $R$ is a domain, then we get that $\si_g(r)=r$ for all $r\in R$. If instead $R$ is $\Z^n$-simple, then note that the set
\[J=\{s\in R\mid (\si_g(r)-r)s=0, \;\;\forall r\in R\}\]
is a non-zero (since $a_gc\in J$) $\Z^n$-invariant ideal of $R$ and hence $1\in J$ which means that
$\si_g=\identity_R$. In both cases we conclude that $g\in K$ and since $g \in \Supp(a)$ was arbitrarily chosen we get $a\in A_K$. This shows that $C_A(R)\subseteq A_K$.
\end{proof}

\begin{corollary}\label{MaxCommConditions}
Let $A=A\TGWdat$ be a TGWA.
If $R$ is maximal commutative in $A$, then $\si:\Z^n\to\Aut_{\K}(R)$ is
injective. Conversely, if $R$ is a domain or $R$ is $\Z^n$-simple, then injectivity of $\si$ implies that $R$ is maximal commutative in $A$.
\end{corollary}

It is an interesting question to ask if the centralizer $C_A(R)$ is
commutative. If this is the case,
then $C_A(R)$ is a maximal commutative subalgebra of $A$, more specifically the
unique maximal commutative subalgebra of $A$ containing $R$.
The following result gives a description of commuting elements in the centralizer.
\begin{theorem}\label{thm:commutantsubalg}
 Let $A=A\TGWdat$ be a regularly graded TGWA.
Let $H$ be any subgroup of $K=\ker(\si)$
of rank one. Then the subalgebra $\bigoplus_{g\in H} A_g$ is commutative.
\end{theorem}
\begin{proof}
Write $H=\Z\cdot g$, where $g\in K$. Let $a,b$ be any homogeneous elements in $A$
of degrees in $H$, say $\deg a = kg$, $\deg b = mg$ where $k,m\in\Z$.
If we can show that $ab-ba=0$, then we are done.
By replacing $g$ by $-g$ we can assume that $k+m\ge 0$.

First we assume that $k=1$. 
Let $c$ be any monic monomial of degree $-g$.
Then $abc^{m+1}$ has degree zero, and thus
\begin{alignat*}{2}
ab c^{m+1} &= \si_g (bc^{m+1}a)  &&\quad\text{by Lemma \ref{lem:nondegtgwaflip}},\\
&=b c^{m+1} a &&\quad\text{since $g\in K$},\\
&=b a c^{m+1} &&\quad\text{since $ca = \sigma_{-g}(ac)=ac$ by Lemma \ref{lem:nondegtgwaflip} and
since $g\in K$}.
\end{alignat*}
By Theorem \ref{thm:regularly_graded_char}\eqref{it:regular_property}, we conclude
that $ab-ba=0$.

Now let $k$ be general.
Let $c$ be any monic monomial of degree $-g$.
Then $abc^{k+m}$ has degree zero, and thus, like before we have
\[ab c^{k+m} = \si_{kg}(b c^{k+m} a) =b c^{k+m} a = b a c^{k+m} \]
where we have used the case $k=1$, $k+m$ times in the last step.
This shows that $(ab-ba)c^{k+m}=0$ 
and by Theorem \ref{thm:regularly_graded_char}\eqref{it:regular_property},
we conclude that $ab-ba=0$.
\end{proof}

Combining Theorem \ref{CommutantDescription} and Theorem \ref{thm:commutantsubalg}
we immediately get the following sufficient condition for the centralizer of $R$ in $A$
to be commutative.
\begin{corollary} \label{cor:commutantsubalg}
Let $A=A\TGWdat$ be a regularly graded TGWA, where $R$ is either a domain or is $\Z^n$-simple.
If $K=\ker(\si)$ has rank at most one, then the centralizer of $R$ in $A$ is commutative.
\end{corollary}
The condition of $K$ having rank at most one is not necessary.
In the next section we give a large family of examples where the
centralizer is commutative, which cover cases for which $K$ has arbitrarily large rank.

The following theorem, which is the main result of this section,
 gives a necessary and sufficient condition for the subalgebra
 $A_K$ of a TGWA $A$ to be commutative. 
It is a generalization of \cite[Lemma 5]{MT_def} (see Remark \ref{rem:MTrelation} below).
\begin{theorem}\label{thm:commutant_suff}
Let $A=A\TGWdat$ be a TGWA and let
$K=\ker(\si)$. The subalgebra $A_K:=\bigoplus_{g\in K} A_g$ is
commutative if and only if there exists a $\Z$-basis
$\{k_1,\ldots,k_s\}$ for $K$ with the following property:
\begin{gather}\label{eq:commutant_suff}
\begin{aligned}
&\text{ For any two basis elements $k_i\neq k_j$, and $m,l\in\Z$, we have $[A_{mk_i},A_{lk_j}]=0$.}
\end{aligned}
\end{gather}
\end{theorem}
\begin{proof}
Clearly \eqref{eq:commutant_suff} is necessary for $A_K$ to be commutative.
Conversely, assume that $\{k_1,\ldots,k_s\}$ is a $\Z$-basis for $K$
such that \eqref{eq:commutant_suff} holds and let $a,b\in A_K$ be arbitrary homogeneous elements.
Then there are unique integers $\alpha_i,\beta_i$ such that
\[
\alpha:=\deg a=\sum_{i=1}^s \alpha_i k_i,\qquad
\beta:=\deg b =\sum_{i=1}^s \beta_i k_i.
\]
For $i \in \{1,\ldots,s\}$, let $a_i'$ be a monic monomial of degree $-\alpha_ik_i$
and $b_i'$ be a monic monomial of degree $-\beta_i k_i$,
and $c_i$ be a monic monomial of degree $\beta_i k_i$.
Put $a'=a_1'\cdots a_s'$ and $b'=b_1'\cdots b_s'$ and $c=c_1\cdots c_s$.
Then using Lemma \ref{lem:nondegtgwaflip} and that $\alpha,\beta\in K$ we have
\begin{equation}\label{eq:suffcalc}
(aba'b')c =\sigma_\alpha(ba'b'a)c = b(a'b'ac)=b\sigma_{-\alpha-\beta}(aca'b')=
baca'b'.
\end{equation}
For any $i \in \{1,\ldots,s\}$, the element $c_i$ commutes with $a_i'$ and $b_i'$
by Theorem \ref{thm:commutantsubalg}, because
all three elements are contained in the subalgebra $\bigoplus_{g\in \Z k_i} A_g$
and $\Z k_i$ is a subgroup of $K$ of rank one.
Also, if $i\neq j$, then $c_i$ commutes with $a_j'$ and with $b_j'$ due
to assumption \eqref{eq:commutant_suff}.
Thus $c$ commutes with both $a'$ and $b'$ which together with \eqref{eq:suffcalc}
entails that $(ab-ba)a'b'c = 0$. 
By Theorem \ref{thm:regularly_graded_char}\eqref{it:regular_property}, this implies
that $ab-ba=0$.
\end{proof}

\begin{remark}\label{Rem:RDomZnSimple}
If $R$ is either a domain or $\Z^n$-simple, then by Theorem \ref{CommutantDescription}, $A_K$ coincides
with the centralizer $C_A(R)$.
\end{remark}
\begin{remark}\label{rem:MTrelation}
Theorem \ref{thm:commutant_suff} is a generalization of \cite[Lemma 5]{MT_def} where
it was proved that $A_K$ is commutative under the assumptions
that $\mu_{ij}=1$ for all $i,j$ and that $X_iX_j=X_jX_i$ for any $i,j\in F(W)$, where
$F(W)$ is the set of $i\in \nn$ such that there is a $k=\sum_{j=1}^n k_je_j \in K$
with $k_i\neq 0$. Even if we only assume that $\mu_{ij}=1$ for all $i\neq j$, this assumption is still stronger
than the one in Theorem \ref{thm:commutant_suff} because there exist examples where
$A_K$ is commutative without satisfying the condition of \cite[Lemma 5]{MT_def}
(see Example \ref{ex:A2}).
\end{remark}

The following is an example where the centralizer $C_A(R)$ of $R$ in
$A$ is not commutative.
\begin{exmp}\label{ex:mu}
Let $n=2$, $R=\K=\C$, $\si_1=\si_2=\identity_\C$, $t_1=t_2=1$ and $\mu_{12}=2$, $\mu_{21}=\frac{1}{2}$.
Then the only consistency relation from \eqref{eq:tgwa_consistency} is $t_1t_2=\mu_{12}\mu_{21}\si_1^{-1}(t_2)\si_2^{-1}(t_1)$, which holds.
Let $A=A\TGWdat$. Then $X_1r=rX_1$ for all $r\in R$ so $R$ is not maximal commutative in $A$.
In fact $A=C_A(R)$, and $A$ is not commutative since $X_1Y_2=2Y_2X_1$.
\end{exmp}

\subsection{Maximal commutative subalgebras of TGWAs associated to Cartan matrices} \label{sec:TqC}
In this section we show that an interesting and more explicit description of the
centralizer $C_A(R)$ is possible for a class of TGWA introduced in \cite{H09}.
Moreover, in all these cases the centralizer is commutative, hence maximal commutative in $A=A\TGWdat$.

Recall that a matrix $C=(a_{ij})_{1\le i,j\le n}$ with integer entries is called a
\emph{generalized Cartan matrix} if the following assertions hold:
\begin{enumerate}[(i)]
	\item $a_{ii}=2$ for all $i \in \nn$;
	\item $a_{ij}\le 0$ for all $i\neq j$;
	\item For all $i,j \in \nn$, $a_{ij}=0$ if and only if $a_{ji}=0$.
\end{enumerate}
Generalized Cartan matrices are fundamental in the theory of Kac-Moody algebras.

In this section we assume that $\K$ is a field.
In \cite{H09} a family of TGWAs were constructed, denoted
 $\mathcal{T}_{q,\mu}(C)$, where $q\in\K\backslash\{0\}$,
$\mu=(\mu_{ij})_{1\le i,j\le n}$ and $C=(a_{ij})_{1\le i,j\le n}$,
a symmetric generalized Cartan matrix.
We will consider the special case when $\mu_{ij}=1$ for all $i,j \in \nn$ and denote
these algebras by $\mathcal{T}_q(C)$. Their construction is as follows.

Take $R$ to be the following polynomial algebra over $\K$:
\[R=\K[H_{ij}^{(k)}\mid 1\le i<j\le n, \text{ and } k=a_{ij}, a_{ij}+2, \ldots, -a_{ij}].\]
Define $\si_1,\ldots,\si_n\in\Aut_{\K}(R)$ by setting, for all $i<j$
and $k=a_{ij}, a_{ij}+2, \ldots, -a_{ij}$:
\begin{subequations}\label{eq:cartantgwa_sigmadef}
\begin{align}
\label{eq:cartantgwa_sigmadef1}
\si_j(H_{ij}^{(k)})&= q^k H_{ij}^{(k)}+H_{ij}^{(k-2)}, \qquad
 \text{where $H_{ij}^{(a_{ij}-2)}:=0$},\\
\label{eq:cartantgwa_sigmadef2}
\si_i(H_{ij}^{(k)})&=\si_j^{-1}(H_{ij}^{(k)}),\\
\si_r(H_{ij}^{(k)})&=H_{ij}^{(k)},\qquad r\neq i,j
\end{align}
\end{subequations}
and define $\si:\Z^n\to\Aut_\K(R)$ by $\si_g=\si_1^{g_1}\circ\cdots\circ\si_n^{g_n}$ for $g\in\Z^n$.
For notational purposes, put
\[H_{ij}=H_{ij}^{(-a_{ij})},\quad
 H_{ji}=\si_j^{-1}(H_{ij}) \quad\text{for all $i<j$,
 \quad and $H_{ii}=1$ for all $i\in\nn$}\]
and define
\[t_i=H_{i1}H_{i2}\cdots H_{in}, \quad \text{for $i\in\nn$.}\]

One can verify that the $\si_i$'s commute with eachother, and that
consistency relations \eqref{eq:tgwa_consistency} hold (see \cite{H09}
for details).
The algebra $\mathcal{T}_q(C)$ is defined as the TGWA associated
to the above data, $\mathcal{T}_q(C)=A\TGWdat$, where
$\mu_{ij}=1$ for all $i,j\in\nn$.

For $q\in\K\backslash\{0\}$,
put $[k]_q=q^{-k+1}+q^{-k+3}+\cdots +q^{k-1}$ if $k\in\Z_{\ge 0}$
and $[-k]_q=-[k]_q$ if $k\in\Z_{<0}$.
Recall that $q$ is said to have \emph{quantum characteristic zero} if,
for any integer $n$,
$[n]_q=0$ implies that $n=0$.

Let $\Gamma_C$ be the Coxeter graph associated to $C$; its vertex
set is $V(\Gamma_C)=\{1,\ldots,n\}$ and $i,j$ are connected if and only if $a_{ij}<0$
(we do not need to label the edges here).
Let $\mathrm{Comp}(\Gamma_C)$ be the set of connected components
of the graph $\Gamma_C$.
 Let $K$ be the kernel of the group homomorphism
$\si:\Z^n\to\Aut_{\K}(R)$. For each $\gamma\in\mathrm{Comp}(\Gamma_C)$
define $g_\gamma\in\Z^n$ by
\[g_{\gamma} = \sum_{i\in V(\gamma)} e_i\]
where $\{e_1,\ldots,e_n\}$ is the standard basis of $\Z^n$, and
$V(\gamma)$ is the vertex set of the subgraph $\gamma$.

Let $A=T_q(C)$. By Theorem \ref{CommutantDescription}, the centralizer $C_A(R)$ is
equal to $A_K=\bigoplus_{g\in K}A_g$.
We have the following description of the gradation group $K$.
\begin{theorem}\label{thm:Kbas}
Assume that $q\in\K\backslash\{0\}$ has quantum characteristic zero.
Let $C$ be an $n\times n$ symmetric generalized Cartan matrix and 
let $A=\mathcal{T}_q(C)=A\TGWdat$ be the twisted generalized Weyl
algebra defined above. Then the set
\begin{equation}\label{eq:Kbas}
\{g_{\gamma}\mid \gamma\in\mathrm{Comp}(\Gamma_C)\}
\end{equation}
forms a $\Z$-basis for the kernel $K$ of $\si:\Z^n\to\Aut_{\K}(R)$.
In particular, the rank of the free 
abelian group $K$ is equal to the number of connected components of the
Coxeter graph $\Gamma_C$ of $C$.
\end{theorem}
\begin{proof} 
Clearly the $g_\gamma$'s are linearly independent over $\Z$.
Let $g=(g_1,\ldots,g_n)\in\Z^n$ and $1\le i<j\le n$.
If $a_{ij}=0$, then $\si_g(H_{ij}^{(0)})=H_{ij}^{(0)}$.
If $a_{ij}<0$, then, by \eqref{eq:cartantgwa_sigmadef},
\begin{align*}
\si_g(H_{ij}^{(2+a_{ij})}) &=
\si_j^{g_j-g_i}(H_{ij}^{(2+a_{ij})})= \\ &=
q^{(2+a_{ij})(g_j-g_i)} H_{ij}^{(2+a_{ij})}+ q^{(a_{ij}+1)(g_j-g_i-1)}[g_j-g_i]_q
 H_{ij}^{(a_{ij})}.
\end{align*}
So $g\in K$ implies that $g_i=g_j$ for all $i,j$ for which
$a_{ij}<0$, i.e. which are connected in $\Gamma_C$. This shows that
any $g\in K$ is a $\Z$-linear combination of the elements $g_\gamma$,
for $\gamma\in\mathrm{Comp}(\Gamma_C)$.

Conversely, if $g\in\mathrm{Comp}(\Gamma_C)$, then $g_i=g_j$
for all $i\neq j$ with $a_{ij}<0$ and thus
$\sigma_g(H_{ij}^{(k)})=\si_j^{g_j-g_i}(H_{ij}^{(k)})=H_{ij}^{(k)}$
for all $i<j$ and all $k$. So $\sigma_g$ fixes all the generators
of $R$ and hence $g\in K$.
\end{proof}

A generalized Cartan matrix is said to be \emph{indecomposable} if it cannot be 
rearranged, by applying simultaneous row and column permutations,
into a block matrix with more than one block.
An immediate corollary of Theorem \ref{thm:Kbas} is that if the Coxeter graph
$\Gamma_C$ of $C$ is connected (which is equivalent to that $C$ is indecomposable)
then $K$ has rank one, and hence by Corollary \ref{cor:commutantsubalg}, $C_A(R)$
is commutative.
The following theorem shows that this holds for any TGWA $A=\mathcal{T}_q(C)$
associated to a symmetric generalized Cartan matrix $C$, not necessarily indecomposable.

\begin{theorem} \label{thm:TCcommutant}
Let $C$ be a symmetric generalized Cartan matrix,
and $q\in\K\backslash\{0\}$ of quantum characteristic zero.
Let $A=\mathcal{T}_q(C)$.
Then the centralizer $C_A(R)$ is commutative (hence maximal commutative in $A$).
\end{theorem}
\begin{proof} We will show that the
$\Z$-basis $\{g_\gamma\mid \gamma\in\mathrm{Comp}(\Gamma_C)\}$ for $K$
satisfies condition \eqref{eq:commutant_suff} of Theorem \ref{thm:commutant_suff}.
Let $\gamma,\gamma'\in\mathrm{Comp}(\Gamma_C)$, $\gamma\neq \gamma'$,
and assume that $a\in A_{m g_\gamma}$, $b\in A_{lg_{\gamma'}}$ for some $m,l\in\Z$.
Then $ab=ba$ due to the fact that
$X_iY_j=Y_jX_i$ for all $i\neq j$, and, by \cite[Theorem 5.2(c)]{H09},
$X_iX_j=X_jX_i$ and $Y_iY_j=Y_jY_i$ for all $i\in V(\gamma), j\in V(\gamma')$.
Thus, by Theorem \ref{thm:commutant_suff}, $C_A(R)$ is commutative.
\end{proof}

\begin{exmp}\label{ex:A2}
Let $C=\left[\begin{smallmatrix}2&-1\\-1&2 \end{smallmatrix}\right]$ and
consider the algebra $\mathcal{T}_q(C)$. The Coxeter graph $\Gamma_C$ is 
\begin{tikzpicture}[baseline=0ex]
\draw (1ex,.7ex)  -- (5ex,.7ex);
\fill (1ex,.7ex) circle (.4ex);
\fill (5ex,.7ex) circle (.4ex);
\end{tikzpicture}.
Thus $K=\Z\cdot (1,1)$,
so $1,2\in F(W)$ in the notation of \cite{MT_def} (see Remark \ref{rem:MTrelation} above).
However, $X_1X_2\neq X_2X_1$ and hence \cite[Lemma 5]{MT_def}
cannot be applied to conclude that the
subalgebra $A_K:=\bigoplus_{g\in \Z^n} A_g$ is commutative.
Nevertheless, by Theorem \ref{thm:TCcommutant}, $A_K$ is commutative.
Furthermore, since $R$ is a domain, $A_K$ coincides with the centralizer
 $C_R(A)$ by Theorem \ref{CommutantDescription}.
\end{exmp}

\section{Ore localizations and a finiteness condition for TGWAs}
\label{sec:Ore}

In this section we introduce a finiteness condition for TGW data, in order
to have suitable Ore sets that allow us to define a well-behaved localization.
The motivation for doing so is that this particular localization turns out
to be useful in the subsequent section when deriving conditions for simplicity
of TGWAs.

\begin{definition}
Let $\TGWdat$ be a TGW datum and let $S$ be a subalgebra of $R$.
We say that $\TGWdat$ is \emph{left $S$-finitistic} if
for any $i,j\in\nn,\, i\neq j,$ there exist some $k\in\Z_{\geq 0}$ and
$s_1,\ldots,s_k\in S$, such that
\begin{equation}\label{eq:finitistic}
\si_i^k(t_j)+s_1\si_i^{k-1}(t_j)+\cdots +s_{k-1}\si_i (t_j) + s_k t_j=0.
\end{equation}
Similarly $\TGWdat$ is called \emph{right $S$-finitistic} if
for any $i,j\in \nn,\, i\neq j$, there exist some $m\in\Z_{\geq 0}$ and
$s_1',\ldots,s_m'\in S$, such that
\begin{equation}\label{eq:right_finitistic}
t_j + s_1'\si_i(t_j) + \ldots + s_m' \si_i^m(t_j)=0.
\end{equation}
$\TGWdat$ is \emph{$S$-finitistic} if it is both left and right $S$-finitistic.
If $\TGWdat$ is $S$-finitistic, then $A\TGWdat$ is also said to be \emph{$S$-finitistic}.
\end{definition}
\begin{remark}
The cases $S=R$ and $S=\K$ were considered in \cite{H09}, with the slight difference
that in \cite{H09} the existence of $s_1,\ldots,s_k$ and $s_1',\ldots,s_m'$ such that
\eqref{eq:finitistic} and \eqref{eq:right_finitistic} hold was required also when $i=j$.
\end{remark}
\begin{prop} Let $\TGWdat$ be a TGW datum.
\begin{enumerate}
\item[{\rm (i)}] If $S_1$ and $S_2$ are subalgebras of $R$ such that $S_1\subseteq S_2$ and if $\TGWdat$ is left (right) $S_1$-finitistic, then $\TGWdat$ is also left (right) $S_2$-finitistic.
\item[{\rm (ii)}] If $R$ is Noetherian, then $\TGWdat$ is $R$-finitistic.
\item[{\rm (iii)}] If $S$ is a subalgebra of $R$ and $\TGWdat$ is $S$-finitistic then for all $i\neq j$,
\begin{multline}\label{eq:finitistic_mij}
\min\{k\in\Z_{\ge 0}\mid \text{equation \eqref{eq:finitistic} holds for some $s_1,\ldots,s_k\in S$}\} = \\= \min\{m\in\Z_{\ge 0}\mid \text{equation \eqref{eq:right_finitistic} holds for some $s_1',\ldots,s_m'\in S$}\}.
\end{multline}
The common number in \eqref{eq:finitistic_mij} is denoted by $m_{ij}$.
\end{enumerate}
\end{prop}
\begin{proof} Part (i) is trivial, and part (ii) was proved in \cite{H09}.
For part (iii), fix $i\neq j$ and assume that \eqref{eq:finitistic} and \eqref{eq:right_finitistic} hold with $k$ and $m$ minimal, but $m\neq k$.
Suppose that $m<k$; the case $m>k$ can be treated analogously. 
 Then by multiplying \eqref{eq:right_finitistic} by $-s_k$ and adding to \eqref{eq:finitistic} we can assume that $s_k=0$. But then we can apply $\si_i^{-1}$ to \eqref{eq:finitistic} and get a contradiction to the minimality of $m$.
\end{proof}
To an $S$-finitistic TGW datum $\TGWdat$ of degree $n$ we associate a matrix $C=C_S\TGWdat=(a_{ij})_{1\le i,j\le n}$ given by
\[a_{ij}=\begin{cases}1-m_{ij},&i\neq j \\ 2,& i=j \end{cases}\]
where $m_{ij}$ was defined in \eqref{eq:finitistic_mij}.

\begin{prop} If $\TGWdat$ is an $S$-finitistic TGW datum where $S$ is a $\Z^n$-invariant subalgebra of $R$ and $t_i \in R_{\reg}$ for each $i\in \nn$, then
$C_S\TGWdat$ is a generalized Cartan matrix.
\end{prop}
\begin{proof}
Since $t_i \in R_{\reg}$ for each $i\in \nn$ we have $m_{ij}\geq 1$ for all $i\neq j$. Hence $a_{ij}\leq 0$ for all $i\neq j$.
Suppose that $a_{ij}=0$ for some $i\neq j$. This means that $m_{ij}=1$ and hence $\sigma_i(t_j)=s t_j$ and $t_j = s' \sigma_i(t_j)$
for some $s, s' \in S$. We may combine these two relations with the regularity of $t_j$ to conclude that $s's=1$.
By \cite[Corollary 2.17]{FutHart} the regularity of the $t_i$'s implies that consistency relations \eqref{eq:tgwa_consistency} hold.
The $\Z^n$-invariance of $S$ in combination with relation \eqref{eq:tgwa_consistency1} and $\sigma_i(t_j)=s t_j$ and $t_j = s' \sigma_i(t_j)$ yields $m_{ji}=1$, i.e. $a_{ji}=0$.
\end{proof}
\begin{remark}
The algebras $T_q(C)$ associated to symmetric generalized Cartan matrices $C$ (see
Section \ref{sec:TqC}) are $\K$-finitistic and have the property that their respective generalized Cartan matrices, as defined above, is precisely $C$. This was the main point
of their construction in \cite{H09}.
\end{remark}

Recall that a \emph{regular left (right) Ore set} $S$ in an algebra $A$ is a multiplicatively
closed subset consisting of regular non-zero elements and containing $1$ such that 
$Sa\cap As\neq \emptyset$ ($aS\cap sA\neq\emptyset$) for all $a\in A$, $s\in S$.
We now prove a theorem which connects the property of a TGWA $A$ being $R$-finitistic to certain natural subsets of $A$ being Ore sets.
\begin{theorem}\label{thm:Ore}
Let $A=A\TGWdat$ be a regularly graded TGWA.
For $i\in\nn$, define the following subsets of $A$:
\begin{equation}
\mathcal{X}_i:=\{X_i^k \mid k\in\Z_{\ge 0}\}, \qquad
\mathcal{Y}_i:=\{Y_i^k \mid k\in\Z_{\ge 0}\}.
\end{equation}
Then the following three assertions are equivalent:
\begin{enumerate}
\item[{\rm (i)}] $\TGWdat$ is left (right) $R$-finitistic;
\item[{\rm (ii)}] $\mathcal{X}_i$ is a regular left (right) Ore set in $A$ for each $i \in \nn$;
\item[{\rm (iii)}] $\mathcal{Y}_i$ is a regular left (right) Ore set in $A$ for each $i \in \nn$.
\end{enumerate}
\end{theorem}
\begin{proof}
We consider the right-sided case and prove that (ii) is equivalent to (i). The other cases are analogous. Let $i\in\nn$ be arbitrary.
It suffices to prove the equivalence of the following two assertion:
\begin{enumerate}[{\rm (a)}]
\item $\mathcal{X}_i$ is a right regular Ore set in $A$;
\item For all $j\neq i$, there exist some $s_1',\ldots,s_m'\in R$ such that
\eqref{eq:right_finitistic} holds.
\end{enumerate}
If $X_i$ is invertible in $A$, then its inverse has to have degree $-e_i$, hence be of the form $r Y_i$ for some $r\in R$, which implies that $t_i$ is invertible. In this case assertions (a) and (b) are both easily seen to hold.
Assume that $X_i$ is not invertible.
 By definition, assertion (a) holds if and only if
\begin{equation}\label{eq:Ore_proof1}
a\mathcal{X}_i\cap sA\neq \emptyset, \qquad \forall a\in A,\;\forall s\in \mathcal{X}_i.
\end{equation}
Since $A$ is generated by $R$, $X_j,Y_j$, for $j \in \nn$, and since
$rX_i=X_i\si_i^{-1}(r)$ for $r\in R$, and
$Y_jX_i=X_i\cdot\mu_{ij}^{-1}Y_j$ for $i\neq j$ and $Y_iX_i^2 = X_i
\si_i^{-1}(t_i)$, it follows that \eqref{eq:Ore_proof1} holds if and only if
\begin{equation}\label{eq:Ore_proof2}
\forall j\neq i,\; \exists m\in\Z_{\ge 0}\quad\text{such that}\quad X_jX_i^{m}\in X_iA.
\end{equation}
Since $A$ is graded and $X_i$ is not invertible, an equality $X_jX_i^{m}=X_ia$ for some $a\in A$ implies that $m>0$ and $a$ is homogeneous of degree $(m-1)e_i+e_j$. For any $m\in\Z_{\ge 0}$ we have  $A_{me_i+e_j}=\sum_{k=0}^m RX_i^kX_jX_i^{m-k}$ and thus \eqref{eq:Ore_proof2} is equivalent to
\begin{equation}\label{eq:Ore_proof3}
\forall j\neq i,\; \exists m\in\Z_{>0}, r_1,\ldots,r_{m}\in R \quad\text{such that}\quad X_jX_i^{m} = \sum_{k=1}^{m} r_k X_i^kX_jX_i^{m-k}.
\end{equation}
By Theorem \ref{thm:regularly_graded_char}\eqref{it:regular_property},
the identity in \eqref{eq:Ore_proof3} is equivalent to (putting $r_0=-1$):
\begin{align*}
0&=\sum_{k=0}^{m} r_k X_i^kX_jX_i^{m-k} Y_jY_i^{m}=\\
&=\Big(\sum_{k=0}^{m} r_k \mu_{ij}^{m-k}\si_i^k\si_j(t_j) \Big)X_i^{m}Y_i^{m} 
\end{align*}
Since $X_i^mY_i^m=\si_i^m(t_i)\cdots\si_i(t_i)$ is a regular element of $R$ it follows that \eqref{eq:Ore_proof3} is equivalent to that assertion (b) holds.
\end{proof}

\begin{corollary}\label{cor:Ore}
If $A=A\TGWdat$ is a regularly graded and $R$-finitistic TGWA, then 
the multiplicative monoid $\mathcal{X}$ in $A$ generated
by $X_1,\ldots,X_n$ is a regular Ore set in $A$.
\end{corollary}
\begin{proof} Straightforward.
\end{proof}

\section{Simplicity theorems for TGWAs}\label{sec:Simplicity}

In this section we provide a description of when TGWAs are simple.

\subsection{A weak simplicity result for general TGWAs}

\begin{theorem}\label{thm:Generical_Simplicity}
Let $A=A\TGWdat$ be a TGWA such that $R$ is $\Z^n$-simple and maximal commutative in $A$.
 If $J$ is a non-zero proper ideal of $A$, then any prime ideal of $R$ containing $J\cap R$,
contains an element of the form $\si_g(t_i)$ for some $i \in \nn$ and $g\in\Z^n$.
\end{theorem}

\begin{proof}
Let $J$ be a non-zero ideal of $A$.
 According to Theorem \ref{IdealIntersection}, $J\cap R$ is a nonzero ideal of $R$.
Suppose that $P$ is a prime ideal of $R$ containing $J\cap R$.
If $P$ would contain $\si_g(J\cap R)$ for all $g\in\Z^n$ then $P$ would also contain the sum $\hat J=\sum_{g\in\Z^n} \si_g(J\cap R)$.
Then $\hat J$ would be a non-zero proper $\Z^n$-invariant ideal of $R$, but this contradicts the $\Z^n$-simplicity of $R$.
Thus there exists a $g\in\Z^n$ such that $\si_g(J\cap R)$ is not contained in $P$. However, $J\cap R$ contains $A_g(J\cap R)A_{-g}=\si_g(J\cap R)A_gA_{-g}$. Since $P$ is a prime ideal we conclude that $A_gA_{-g}\subseteq P$. Choosing two monic monomials $a\in A_g$ and $b\in A_{-g}$, the product $ab$ can be written as a product of elements of the form $\si_h(t_i)$, where $h\in\Z^n$ and $i\in\nn$. Since $P$ is prime it must contain at least one such factor, which proves the theorem.
\end{proof}

\begin{corollary}\label{cor:Generical_Simplicity}
Let $A=A\TGWdat$ be a TGWA.
Consider the following countable union of Zariski-closed sets in $\Spec(R)$:
\[S=\bigcup_{\substack{g\in\Z^n \\ i \in \nn}} V\big(\si_g(t_i)\big).\]
If $R$ is $\Z^n$-simple and maximal commutative in $A$, then $S$ contains
$\bigcup_{J \lhd A}V(J\cap R)$, i.e. the union of varieties of the ideals of the form $J\cap R$ where $J$ ranges over the ideals of $A$.
\end{corollary}

\begin{remark}
Corollary \ref{cor:Generical_Simplicity} can be interpreted as saying that there are ``few'' proper ideals of $A$, in other words $A$ is ``close'' to being simple. In particular, if $R$ is $\Z^n$-simple and maximal commutative in $A$ and in addition $t_i$ is invertible for each $i\in \nn$, then $A$ is simple.
This will be made more precise in Corollary \ref{cor:explained_simplicity}.
\end{remark}

\subsection{On the structure of the localized algebra $\mathcal{X}^{-1}A$}
 \label{sec:The_algebra_B}
Let $A=A\TGWdat$ be a regularly graded TGWA which is $R$-finitistic. By Corollary \ref{cor:Ore}, the multiplicative submonoid $\mathcal{X}$ of $A$,
generated by $X_1,\ldots,X_n$, is an Ore set in $A$. Let $B=\mathcal{X}^{-1}A$. 
Since $\mathcal{X}$ consists of regular elements in $A$, the canonical map $A\to B$
is injective and we can, and will henceforth, regard $A$ as a subalgebra of $B$.
In this section we shall prove a key result about the algebra $B$ (Theorem \ref{thm:B_intersection}) which will later allow us to deduce the simplicity criterion for $A$ in Section \ref{SimplicityOfFinitisticTGWAs}. 
\begin{remark}
The algebra $B$ can be embedded into another localization of $A$ which is a $\Z^n$-crossed product.
Let $T$ be the multiplicative submonoid of $R$ generated by the set $\{\si_g(t_i)\mid g\in\Z^n, \,\, i \in \nn \}$.
By \cite[Theorem 2.15]{FutHart}, $T$ is an Ore set in $A$ and the localization $T^{-1}A$ is a crossed product, isomorphic to the TGWA $A(T^{-1}R,\si,\vect{t},\mu)$, where $\si$ is uniquely extended to a $\Z^n$-action on $T^{-1}R$.
There is a $\K$-algebra monomorphism $\tau: B\to T^{-1}A$
given by $\tau(X_i)=X_i$, $\tau(X_i^{-1})=t_i^{-1} Y_i$ and $\tau(r)=r$ for all $r\in R$ and $i\in \nn$.
\end{remark}

Before we can continue we need to prove two lemmas,
which will be of vital importance in the proof of our main results.

\begin{lemma}\label{lem:B_properties_liten}
Let $A=A\TGWdat$ be a regularly graded and $R$-finitistic TGWA and put $B=\mathcal{X}^{-1}A$.
The following assertions hold:
\begin{enumerate}[{\rm (i)}]
	\item\label{it:Bprop1} For any non-zero element $b\in B$,
    there exists an $r\in R_{\reg}$ such that $rb\in A$;
  \item\label{it:Bprop3} $R_{\reg}\subseteq B_{\reg}$;
  \item\label{it:Bprop6} $Z(A)\subseteq Z(B)$;
  \item\label{it:Bprop8} If $AxA=A$ for all $x\in\mathcal{X}$, then $Z(A)=Z(B)$.
\end{enumerate}
\end{lemma}
\begin{proof}
 (i) Any element $b\in B$ can be written as $b=X_{i_1}^{-1}\cdots X_{i_k}^{-1}a$
 where $a\in A$. But $Y_iX_i=t_i$ so $t_iX_i^{-1}=Y_i$ for $i\in \nn$. Therefore, multiplying
 $b$ from the left by the regular element
 $r=(\si_{i_1}\circ\cdots\circ\si_{i_{k-1}})^{-1}(t_{i_k})\cdots\si_{i_1}^{-1}(t_{i_2})t_{i_1}$
 we get $Y_{i_1}\cdots Y_{i_k}a$, which is an element of $A$.
 
 (ii) Let $r\in R_{\reg}$ and $b\in B\backslash\{0\}$ be arbitrary.
 Seeking for a contradiction, assume that $rb=0$. Without loss
 of generality we may assume that $b$ is homogeneous.
 Since $B=\mathcal{X}^{-1}A$, there is an $x\in\mathcal{X}$ with $0\neq xb\in A$.
 By Corollay \ref{cor:gamma_nondeg} there is
 a homogeneous $c\in A$ such that $xbc\in R\setminus\{0\}$. Since 
 $r\in R_{\reg}$ we get $\si_g(r)xbc\neq 0$, where $g=\deg(x)$. However,
 $\si_g(r)xbc=xrbc=0$, which is a contradiction. This shows that $r\in B_{\reg}$.

 (iii) Let $a\in Z(A)$ and $b\in B$. Write $b=x^{-1}a_1$, where $x\in\mathcal{X}$
 and $a_1\in A$. Then $ax=xa$ and $aa_1=a_1a$ imply $ab=ba$. Thus $a\in Z(B)$.

(iv) Let $b\in Z(B)$. Since $A\subseteq B$ it is enough to show that $b \in A$.
Since $B=\mathcal{X}^{-1}A$, we have
 $b=x^{-1}a$ for some $x\in\mathcal{X}, a\in A$.
 By the assumption $1=\sum_i c_ixd_i$ for some $c_i,d_i\in A$.
 Now note that, since $b$ commutes with any element of $A$,
 \[A\ni \sum_i c_iad_i = \sum_i c_ixbd_i = \sum_i c_ixd_i b=b.\]
 This proves that $b\in A$. Hence $b\in Z(A)$, so $Z(B)\subseteq Z(A)$.
 The converse inclusion was shown in \eqref{it:Bprop6}.
\end{proof}
\begin{remark}
Under the same assumptions as in Lemma \ref{lem:B_properties_liten}
one may also prove the following statements:
(i) For any non-zero element $b\in Z(B)$, 
    there is a regular element $r\in R$ such that $rb\in Z(C_A(R))$;
(ii) $B_0$ is commutative;
(iii) $C_A(B_0)=C_A(R)$;
(iv) If $R$ is $\Z^n$-simple, then $Z(B)\cap B_0\subseteq R$.
However, this paper makes no use of these facts, 
and therefore we omit the proof.
\end{remark}

The second lemma that we need is a technical step used in the proof of Theorem \ref{thm:B_intersection} below.
Note that the $\Z^n$-gradation on $A$ can be extended to a $\Z^n$-gradation on $B$ by putting $\deg(X_i^{-1})=-e_i$ for each $i\in\nn$.

\begin{lemma}\label{lem:B_replacement}
Let $A=A\TGWdat$ be a regularly graded and $R$-finitistic TGWA such that $R$ is $\Z^n$-simple, and put $B=\mathcal{X}^{-1}A$.
For any non-zero $b\in B$ there exists an element $b'\in B$ with the following properties:
\begin{enumerate}
\item[{\rm (i)}] $b'\in BbB$;
\item[{\rm (ii)}] $(b')_0 = 1$, where $(b')_0$ is the degree zero component of $b'$
with respect to the $\Z^n$-gradation on $B$;
\item[{\rm (iii)}] $|\Supp(b')|\le |\Supp(b)|$.
\end{enumerate}
\end{lemma}
\begin{proof}
 Let $b\in B$ be non-zero. Then $b=\sum_{g\in\Z^n}b_g$ where $b_g\in B_g$, for $g\in \Z^n$ and we may choose some $h\in \Z^n$ such that $b_h\neq 0$.
By Theorem \ref{thm:regularly_graded_char}\eqref{it:regular_property}, we know that $b_h c\neq 0$ for any monic monomial $c$ of degree $-h$. Thus, by replacing $b$ by $bc$ we can without loss of generality assume that $b_0\neq 0$. By Lemma \ref{lem:B_properties_liten}\eqref{it:Bprop1}, there is an $r\in R_{\reg}$ such that $rb\in A$. By Lemma \ref{lem:B_properties_liten}\eqref{it:Bprop3} the element $rb_0$ is non-zero.
The set
\[J=\big\{s\in R\mid s + \sum_{g\in\Supp(b)\setminus\{0\}}c_g\in BbB \text{ for some $c_g\in B_g$}\big\}\]
contains the non-zero element $r b_0$ (take $c_g=rb_g$ for $g\in \Z^n$) and hence $J$ is a non-zero ideal of $R$. We shall now show that $J$ is $\Z^n$-invariant. If $s\in J$, then $s+\sum_{g\in\Supp(b)\setminus\{0\}}c_g\in BbB$ for some $c_g\in B_g$. So for any $i\in\nn$,
\[BbB\ni X_i (s+\sum_{g\in\Supp(b)\setminus\{0\}} c_g) X_i^{-1} = \si_i(s) + \sum_{g\in\Supp(b)\setminus\{0\}}X_ic_gX_i^{-1}\]
which shows that $\si_i(s)\in J$. Similarly $\si_i^{-1}(s)\in J$. Thus $J$ is $\Z^n$-invariant. Since $R$ is assumed to be $\Z^n$-simple we deduce that $J=R$. Hence $1\in J$, which means that there are $c_g\in B_g$ such that $b':=1+\sum_{g\in\Supp(b)\setminus\{0\}}c_g\in BbB$. This $b'$ satisfies the required properties.
\end{proof}

Now we come to the main result about the algebra $B$ which in particular implies
that the center $Z(B)$ of $B$ is an essential subalgebra of $B$, in the sense
of Defininition \ref{def:essential_subalgebra}.
\begin{theorem}\label{thm:B_intersection}
Let $A=A\TGWdat$ be a regularly graded and $R$-finitistic TGWA such that $R$ is $\Z^n$-simple, and put $B=\mathcal{X}^{-1}A$.
Every non-zero ideal of $B$ has non-empty intersection with $Z(B)\cap \big(1+\sum_{g\in\Z^n\setminus\{0\}} B_g\big)$. In particular, $Z(B)$
is an essential subalgebra of $B$.
\end{theorem}
\begin{proof}
 Let $J\subseteq B$ be a non-zero ideal. Among all non-zero elements of $J$, let $b\in J$ be one such that $|\Supp(b)|$ is as small as possible. By Lemma \ref{lem:B_replacement} we can construct an element $b'$ with $b'\in BbB\subseteq J$, $(b')_0=1$, and $|\Supp(b')|\le |\Supp(b)|$. In fact, by minimality of $|\Supp(b)|$ among all non-zero elements of $J$, we have $|\Supp(b')|=|\Supp(b)|$. Let $r\in R$ be arbitrary. Then $(b'r-rb')_0 = r-r=0$ so $|\Supp(b'r-rb')|<|\Supp(b')|$. By minimality of $|\Supp(b')|$ and that $b'r-rb'\in J$ we must have $b'r-rb'=0$. Let $i\in\nn$. Then $(X_i b' X_i^{-1} - b')_0 = 1-1=0$. Again $X_ib'X_i^{-1}-b'\in J$ so by minimality of $|\Supp(b')|$ we conclude that $X_ib'X_i^{-1}-b'=0$. So $X_ib'=b'X_i$ for all $i \in \nn$. Since $B$ is generated as a ring by the elements of $R$ and $X_i, X_i^{-1}$, for $i \in \nn$, we conclude that $b'\in J\cap Z(B)\cap \big(1+\sum_{g\in\Z^n\setminus\{0\}} B_g\big)$.
\end{proof}

\subsection{Simplicity of finitistic TGWAs}\label{SimplicityOfFinitisticTGWAs}

In this section, let $A=A\TGWdat$ be a regularly graded and
$R$-finitistic TGWA. Recall that $\mathcal{X}$
is the multiplicative submonoid in $A$ generated by $X_1,\ldots,X_n$, which by Theorem \ref{thm:Ore} is a regular Ore set in $A$,
and put $B=\mathcal{X}^{-1}A$.

Our first theorem in this section
reduces the question of the simplicity of a TGWA to the simplicity of the localized algebra $B$. This method was inspired by the ideas of Jordan \cite[Theorem 6.1]{J93}, \cite[Theorem 3.2]{J95}.

\begin{theorem}\label{thm:Simplicity_of_TGWAs_using_B}
Let $A=A\TGWdat$ be a regularly graded and $R$-finitistic TGWA.
Then $A$ is simple if and only if the following two assertions hold:
\begin{enumerate}[{\rm (i)}]
\item $B$ is simple;
\item $AxA=A$ for all $x\in \mathcal{X}$.
\end{enumerate}
\end{theorem}
\begin{proof}
Suppose that $A$ is simple. Let $J$ be any non-zero ideal of $B$ and let $0\neq b\in J$.
Since $B=\mathcal{X}^{-1}A$ there is an $x\in \mathcal{X}$ such that $xb\in A$.
Thus $0\neq xb\in A\cap J$ so $A\cap J$ is a non-zero ideal of $A$, thus $1\in J$
since $A$ is simple. This shows that (i) holds.
Since $A$ is regularly graded, each element $x\in\mathcal{X}$
 is non-zero by Theorem \ref{thm:regularly_graded_char}. So, since $A$ is simple, (ii) must hold.

For the converse, assume that condition (ii) holds. Then we prove something slightly stronger than that (i) implies $A$ is simple. Namely, we show that whenever $J$ is an ideal of $A$ such that $BJB=B$, then $J=A$. So suppose $J\subseteq A$ is an ideal with $BJB=B$. Using that $\mathcal{X}$ is an Ore set in $A$, it is straightforward to show that $BJB=\big\{x^{-1}a\mid x\in\mathcal{X}, a\in J\big\}$. Since $BJB=B$, we have $1=x^{-1}a$ for some $x\in\mathcal{X}, a\in J$. Then $x=a\in J$. By condition (ii), $A= AxA=AaA\subseteq J$. Thus $J=A$.
\end{proof}

In what follows, it is useful to keep the following facts in mind.
Note that the assumption of $R$-finitisticity is not required here.

\begin{lemma}\label{lem:center_field}
Let $A=A\TGWdat$ be a regularly graded TGWA. Consider the following assertions:
\begin{enumerate}[{\rm (i)}]
\item $Z(A)$ is a field;
\item $Z(A)\subseteq R$;
\item $Z(A)=R^{\Z^n}:=\big\{r\in R\mid \si_g(r)=r,\;\forall g\in\Z^n\big\}$.
\end{enumerate}
Then {\rm (i)}$\Rightarrow${\rm (ii)}$\Rightarrow${\rm (iii)}.
If $R$ is $\Z^n$-simple, then all three assertions are equivalent.
\end{lemma}
\begin{proof}
(i)$\Rightarrow$(ii): Suppose that $Z(A)_g\neq \{0\}$ for some $g\in\Z^n\setminus\{0\}$
and let $0\neq a\in Z(A)_g$. Then $1+a\in Z(A)$ and hence it is invertible.
Using that $\Z^n$ is an orderable group,
fix an ordering $<$ on $\Z^n$. Without loss of generality we may assume that $g>0$.
Let $b$ be the inverse of $1+a$.
Write $b=b_{h_1}+\cdots+b_{h_k}$ where $0\neq b_{h_i}\in B_{h_i}$, $h_i\in \Z^n$, for all $i\in \{1,\ldots,k\}$ and $h_1<\cdots <h_k$.
In the product $(1+a)b$, the term of lowest degree is $1b_{h_1}$ and
the term of highest degree is $ab_{h_k}$ which is non-zero since $a$ is invertible.
On the other hand, $(1+a)b=1$. Thus $k=1$ and $b_{h_1}=ab_{h_1}=1$ which 
contradicts that $g+h_1>h_1$. Hence $Z(A)\subseteq R$.

(ii)$\Rightarrow$(iii): Let $r\in Z(A)$. Then $rX_i=X_ir$ for any $i\in\nn$.
Equivalently, $(r-\si_i(r))X_i=0$ for any $i\in\nn$.
By Theorem \ref{thm:regularly_graded_char}\eqref{it:torsionfree}
we get $\si_i(r)=r$ for all $i\in\nn$. Thus $\si_g(r)=r$ for each $g\in\Z^n$.
The converse inclusion is straightforward.

Assume that $R$ is $\Z^n$-simple. It is enough to prove that
(iii)$\Rightarrow$(i).
Let $r\in Z(A)$ be non-zero. Then, by (iii), $Rr$ is a non-zero $\Z^n$-invariant
ideal of $R$, hence it contains $1$ and thus $r$ is invertible.
\end{proof}

The following result holds for an arbitrary TGWA.

\begin{lemma}\label{SimpleImpliesZnSimple}
Let $A=A\TGWdat$ be a TGWA. If $A$ is simple, then $R$ is $\Z^n$-simple.
\end{lemma}

\begin{proof}
Suppose that $J$ is a proper $\Z^n$-invariant ideal of $R$. Then, since $R=A_0$,
\[(AJA)\cap R = \sum_{g\in\Z^n}A_gJA_{-g} =
   \sum_{g\in \Z^n}\si_g(J)A_gA_{-g}\subseteq JR\subseteq J. \]
Hence $AJA$ is a proper ideal of $A$. Since $A$ is simple it follows that $AJA=\{0\}$, i.e. $J=\{0\}$. This proves that $R$ is $\Z^n$-simple.
\end{proof}

The next result is the main theorem of this section and provides
necessary and sufficient conditions for an $R$-finitistic TGWA to be simple.
\begin{theorem} \label{thm:Simplicity_of_TGWAs}
Let $A=A\TGWdat$ be a regularly graded and $R$-finitistic TGWA.
Then $A$ is simple if and only if the following three assertions hold:
\begin{enumerate}[{\rm (i)}]
\item \label{it:mainthm_Jcond}
   $AxA=A$ for all $x\in\mathcal{X}$;
\item \label{it:mainthm_Zn-simplicity}
   $R$ is $\Z^n$-simple;
\item \label{it:mainthm_center}
   $Z(A)\subseteq R$.
\end{enumerate}
\end{theorem}
\begin{proof}
Suppose that $A$ is simple. By Theorem \ref{thm:Simplicity_of_TGWAs_using_B}, condition \eqref{it:mainthm_Jcond} holds.
Lemma \ref{SimpleImpliesZnSimple} shows that \eqref{it:mainthm_Zn-simplicity} holds.
Simplicity of $A$ also implies that $Z(A)$ is a field. Hence \eqref{it:mainthm_center} follows from Lemma \ref{lem:center_field}.

For the converse, assume that (i)--(iii) hold. By \eqref{it:mainthm_Jcond} and Theorem \ref{thm:Simplicity_of_TGWAs_using_B} it is enough to prove that $B$ is simple. Let $J$ be any non-zero ideal of $B$. Since \eqref{it:mainthm_Zn-simplicity} holds, Theorem \ref{thm:B_intersection} applies to give $J\cap Z(B)\cap\big(1+\sum_{g\in\Z^n\setminus\{0\}}B_g\big)\neq \emptyset$. By assumptions \eqref{it:mainthm_Jcond} and \eqref{it:mainthm_center}, Lemma \ref{lem:B_properties_liten} \eqref{it:Bprop8} implies that $Z(B)\subseteq R$. Hence $Z(B)\cap \big(1+\sum_{g\in\Z^n\setminus\{0\}}B_g\big)=\{1\}$. Thus $1\in J$ and therefore $J=B$.
\end{proof}

\begin{remark}
In some cases assertion (i) can be made more explicit, see Section \ref{sec:TypeA1n}.
\end{remark}

\begin{corollary}\label{cor:explained_simplicity}
If $A=A\TGWdat$ is a regularly graded and $R$-finitistic TGWA
where $R$ is $\Z^n$-simple and maximal commutative in $A$, then $A$ is simple
if and only if $AxA=A$ for all $x\in\mathcal{X}$.
\end{corollary}
\begin{remark}
Corollary \ref{cor:explained_simplicity} can be compared to the ``generic'' result that we proved earlier in Corollary \ref{cor:Generical_Simplicity}, where it was shown that if $R$ is $\Z^n$-simple and maximal commutative in $A$, then in some sense $A$ has ``few'' ideals. Corollary \ref{cor:explained_simplicity} explains this somewhat, in that it shows that the only ideals in $A$ which obstruct the simplicity when $R$ is $\Z^n$-simple and maximal commutative in $A$ are those of the form $AxA$ for $x\in\mathcal{X}$.
See Example \ref{ex:nonsimpleGWA} for a TGWA $A=A\TGWdat$ where $R$ is $\Z^n$-simple and maximal commutative in $A$, but where $AX_1A$ is a proper non-zero ideal, and hence $A$ is not simple.
\end{remark}

\subsection{Simplicity of TGWAs of Lie type $A_1\times\cdots\times A_1$}
\label{sec:TypeA1n}

\begin{definition}
If $A\TGWdat$ is a regularly graded and $R$-finitistic TGWA, whose associated generalized Cartan matrix
is of a certain type $X$ then $A\TGWdat$ is said to be of \emph{Lie type} $X$.
\end{definition}

Let $A=A\TGWdat$ be a regularly graded, $R$-finitistic TGWA. We assume in this section that $A$ is of Lie type $(A_1)^n=A_1\times\cdots\times A_1$. This is equivalent to that for all $i,j\in\nn$ there exist invertible $r_{ij},s_{ij}\in R$ such that $X_iX_j=r_{ij}X_jX_i$ and $Y_iY_j=s_{ij}Y_iY_j$. This case covers all higher rank generalized Weyl algebras (such that $t_i\in R_{\reg}$ for all $i\in \nn$). Indeed, they correspond to the case $r_{ij}=s_{ij}=1$ for all $i,j \in \nn$. Furthermore, the TGWAs constructed in \cite{MPT} (certain Mickelson-Zhelobenko algebras and Gelfand-Tsetlin algebras) are of Lie type $(A_1)^n$ as well as some examples of so called crystalline graded rings.

The following result shows how assertion \eqref{it:mainthm_Jcond} of Theorem \ref{thm:Simplicity_of_TGWAs} can be made more explicit in the case of TGWAs of Lie type $(A_1)^n$.

\begin{lemma} \label{lem:AxA_vs_J}
Let $A=A\TGWdat$ be a regularly graded and $R$-finitistic TGWA of Lie type $A_1\times\cdots\times A_1$.
Then the following assertions are equivalent:
\begin{enumerate}[{\rm (i)}]
\item $AxA=A$ for all $x\in\mathcal{X}$;
\item $Rt_i+R\si_i^d(t_i)=R$, for all $i\in\nn$ and $d\in\Z_{>0}$.
\end{enumerate}
\end{lemma}
\begin{proof}
First we show that
\[ AxA=A,\;\forall x\in\mathcal{X}\Longleftrightarrow AX_i^dA=A,\;\forall i\in\nn,\forall d\in\Z_{>0}.\]
The implication $\Rightarrow$ is trivial since $X_i^d\in\mathcal{X}$ for all $i,d$.
Conversely, assume that $AX_i^dA=A$ for all $i,d$ and let $x\in\mathcal{X}$.
Since $A$ is of Lie type $(A_1)^n$, we can reorder the factors in $x$
(up to multiplicative invertible factors from $R$) and thus get $AxA=Ax'A$ where
 $x'=X_1^{j_1}\cdots X_n^{j_n}$ for some $j_r\ge 0$.
By assumption, $AX_1^{j_1}A=A$. Thus
\begin{equation}\label{eq:A1proof_steg75}
1\in (AX_1^{j_1}A)\cap R=\sum_{g\in \Z^n} A_gX_1^{j_1}A_{-g-j_1e_1}.
\end{equation} 
Since $A$ is of Lie type $(A_1)^n$, for each $g\in \Z^n$ we have $A_g=RZ_1^{(g_1)}\cdots Z_n^{(g_n)}$.
Also $RZ_i^{(k)}Z_j^{(l)}\subseteq R Z_j^{(l)}Z_i^{(k)}$ for all $i\neq j$ and any $k,l\in\Z$.
Thus, \eqref{eq:A1proof_steg75} implies that
\[1=\sum_{k\in\Z} r_kZ_1^{(k)}X_1^{j_1} Z_1^{(-k-j_1)}\]
for some $r_k\in R$, $k\in \Z$, (only finitely many non-zero).
We can choose $s_k\in R$, $k\in \Z$, such that \[\sum_{k\in\Z} s_k Z_1^{(k)}x'Z_1^{(-k-j_1)}=
\sum_{k\in\Z} r_k Z_1^{(k)}X_1^{j_1} Z_1^{(-k-j_1)}\cdot X_2^{j_2}\cdots X_n^{j_n}
=X_2^{j_2}\cdots X_n^{j_n}.\]
This proves that $X_2^{j_2}\cdots X_n^{j_n}\in Ax'A=AxA$. Replacing $x$ by
$X_2^{j_2}\cdots X_n^{j_n}$ and repeating this procedure for $2,3,\ldots, n$ yields $1\in AxA$, hence $A=AxA$.

It remains to show that
 $AX_i^dA=A,\;\forall i\in\nn,\forall d\in\Z_{> 0}$
 is equivalent to $Rt_i+R\si_i^d(t_i)=R,\;\forall i\in\nn,\forall d\in\Z_{> 0}$.
Similar to the reasoning above we have for any $d>0$ and any $i\in\nn$,
 $AX_i^dA\cap RX_i^{d-1}=\sum_{k\in\Z}RZ_i^{(k)}X_i^dZ_i^{(-k-1)}=RY_iX_i^d+RX_i^dY_i
 =\big(Rt_i+R\si_i^d(t_i)\big)X_i^{d-1}$ which implies the required equivalence.
\end{proof}

Lemma \ref{lem:AxA_vs_J} and Theorem \ref{thm:Simplicity_of_TGWAs} immediately imply the following result.

\begin{theorem}\label{thm:Simplicity_of_TGWAs_of_type_A1^n}
Let $A=A\TGWdat$ be a regularly graded TGWA which is $R$-finitistic of Lie type $(A_1)^n$.
Then $A$ is simple if and only if the following three assertions hold:
\begin{enumerate}[{\rm (i)}]
\item $R\si_i^d(t_i)+Rt_i = R$, for all $i\in\nn$ and $d\in\Z_{>0}$;
\item $R$ is $\Z^n$-simple;
\item $Z(A)\subseteq R$.
\end{enumerate}
\end{theorem}

The next result deals with the case of higher rank generalized Weyl algebras. Taking $n=1$ in the following theorem, we recover the case of \cite[Theorem~6.1]{J93}.
\begin{theorem}\label{thm:Simplicity_of_GWAs}
If $A=R(\si,\vect{t})$ is a generalized Weyl algebra of rank $n$, then $A$ is simple if and only if the following four assertions hold:
\begin{enumerate}[{\rm (i)}]
\item $t_i \in R_{\reg}$ for all $i\in\nn$;
\item\label{it:RtRsi} $Rt_i+R\si_i^d(t_i)=R$ for $i\in \nn$ and all $d\in\Z_{>0}$;
\item $R$ is $\Z^n$-simple;
\item\label{it:GWA_injectivity} $\si:\Z^n\to\Aut_\K(R)$ is injective.
\end{enumerate}
\end{theorem}
\begin{proof}
If $A$ is simple, then all $t_i$'s must be regular. Indeed, if $rt_i=0$ for some $i\in \nn$, where $0\neq r\in R$, then one can check that the element $rY_i$ of $A$ (which is non-zero in the present case of GWAs) generates an ideal which is proper: $(ArY_iA)\cap R=\{0\}$. Also, if all $t_i$'s are regular, it is well-known that $A$ is isomorphic to the TGWA $A\TGWdat$, where $\mu_{ij}=1$ for all $i,j$ (see \cite[Example 1.2]{MT_def}) and that it is $R$-finitistic and of Lie type $(A_1)^n$. Thus we can apply Theorem \ref{thm:Simplicity_of_TGWAs_of_type_A1^n}. It only remains to prove that $Z(A)\subseteq R$ if and only if $\si$ is injective. For this, the key point is that if $g\in K=\ker(\si)$, then $Z(A)_g=R^{\Z^n}Z_1^{(g_1)}\cdots Z_n^{(g_n)}$ (here it is crucial that $A$ is a GWA). Then it is clear that $K=\{0\}$ if and only if $Z(A)\subseteq R$.
\end{proof}

Combining Theorem \ref{thm:Simplicity_of_GWAs} and Corollary \ref{MaxCommConditions} we get the following.

\begin{corollary}\label{cor:SimplicityMaxComm_of_GWAs}
If $A=R(\si,\vect{t})$ is a simple generalized Weyl algebra of rank $n$, then $R$ is a maximal commutative subalgebra of $A$.
\end{corollary}

\begin{remark}
If $A=R(\si,\vect{t})$ is a GWA of rank $n$ where $t_i$ is non-invertible for each $i\in\nn$, then condition \eqref{it:GWA_injectivity} of Theorem \ref{thm:Simplicity_of_GWAs} can be removed. 
Indeed, then $\si:\Z^n\to\Aut_\K(R)$ must be injective if
$Rt_i+R\si_i^d(t_i)=R$ for all $i\in\nn$ and all $d\in\Z_{>0}$, which follows from
the fact that, for GWAs, $\si_i(t_j)=t_j$ for $i\neq j$.
\end{remark}

\section{Examples}\label{Sect:Examples}
The following example shows a simple TGWA where $R$ is not maximal commutative.
\begin{exmp}\label{ex:muSimple}
Let $A=A\TGWdat$ be as in Example \ref{ex:mu}, i.e.
 $n=2$, $R=\K=\C$, $\si_1=\si_2=\identity_\C$, $t_1=t_2=1$ and $\mu_{12}=2$, $\mu_{21}=\frac{1}{2}$. Note that in $A$ we have $Y_iX_i=1=X_iY_i$ for $i=1,2$.
We will show that $Z(A)\subseteq R$. Let $a\in Z(A)$ be non-zero. Since $Z(A)$ is a graded subalgebra of $A$ we can assume that $a$ is homogeneous. Let $g=(g_1,g_2)\in\Z^2$ be the degree of $a$. Since $X_2X_1=2X_1X_2$ (since $Y_2=X_2^{-1}$) we have $a=rX_1^{g_1}X_2^{g_2}$ for some non-zero $r\in R$. Therefore $0=[a,X_1]=(2^{g_2}-1)rX_1^{g_1+1}X_2^{g_2}$ which implies $g_2=0$. Similarly $g_1=0$, which proves that $Z(A)\subseteq R$. Trivially $R=\C$ is $\Z^n$-simple and $AxA=A$ for all $x\in\mathcal{X}$ (since $X_1$ and $X_2$ are invertible), hence Theorem  \ref{thm:Simplicity_of_TGWAs} implies that $A$ is simple.
\end{exmp}

Condition \eqref{it:mainthm_Jcond} in Theorem \ref{thm:Simplicity_of_TGWAs} and condition \eqref{it:RtRsi} in Theorem \ref{thm:Simplicity_of_GWAs}
are not superfluous, as the following example shows.

\begin{exmp}\label{ex:nonsimpleGWA}
Let $n=1$, $\K=\C$, $R=\C[u]$, $\si_1(u)=u+1$, $t=u(u-1)$ and $A=R(\si,\vect{t})$. 
Then $R$ is a maximal commutative subalgebra of $A$, and $R$ is $\Z$-simple.
However, $Rt+R\si_1(t)=Ru$ which is a proper ideal of $R$. Hence, by
Theorem \ref{thm:Simplicity_of_GWAs}, $A$ is not simple.
\end{exmp}

The construction of the following TGWA is due to A. Sergeev \cite[Section~1.5.2]{Sergeev}, but we consider $\mathfrak{sl}(n+1,\C)$ instead of $\mathfrak{gl}(n+1,\C)$.
\begin{exmp}
Let $\mathfrak{h}$ be the Cartan subalgebra consisting of all traceless diagonal matrices in the Lie algebra $\mathfrak{sl}(n+1,\C)$. 
For $i \in \{1,\ldots,n+1\}$, let $\ep_i\in\mathfrak{h}^\ast$ be given by $\ep_i(\diag(\la_1,\ldots,\la_{n+1}))=\la_i$. For $i\in \nn$ let $\al_i:=\ep_i-\ep_{i+1}$ be the simple roots, and let $h_i\in\mathfrak{h}$ be such that $\al_i(h_j)=\delta_{ij}$. Let $f_1,\ldots,f_{n+1}\in\C[u]$ be an arbitrary collection of polynomials in one indeterminate $u$.

Take $\K=\C$ and let $R=S_\C(\mathfrak{h})$, the symmetric algebra on $\mathfrak{h}$. For $i\in \nn$, define $\si_i\in\Aut_\C(R)$ by requiring that $\si_i(h)=h-\al_i(h),\;\forall h\in\mathfrak{h}$. Let $\mu_{ij}=1$ for all $i,j$ and put
 \[t_1=f_1(h_1)f_2(h_2-h_1),\,\, t_2=f_2(h_2-h_1+1)f_3(h_3-h_2), \,
    \ldots \, , \, t_n=f_n(h_n-h_{n-1}+1)f_{n+1}(h_n).\]
One can verify that the consistency relations \eqref{eq:tgwa_consistency} hold. Let $\mathcal{S}(f_1,\ldots,f_{n+1})=A\TGWdat$ be the corresponding TGWA.
One can show that it is $R$-finitistic.
 We will use Theorem \ref{thm:Simplicity_of_TGWAs} to prove the following result.
\begin{prop}
$\mathcal{S}(1,u,1)$ is simple.
\end{prop}
\begin{proof}
Regard $R$ as a polynomial algebra over $\C$ in variables $h_1,h_2$.
We show that $R$ is $\Z^2$-simple. Let $J$ be any non-zero $\Z^2$-invariant
ideal of $R$. Among all non-zero elements $f$ of $J$, choose one whose
$h_1$-degree is minimal. If it is positive, then $f-\si_1(f)$ has 
smaller degree, hence $f-\si_1(f)=0$ (otherwise
$f-\si_1(f)$ would be a non-zero element of $J$ since $J$ is $\Z^2$-invariant,
which would contradict the choice of $f$). If the $h_1$-degree of $f$ is zero
then we get $\si_1(f)=f$ anyway. Next, among all non-zero elements $f$ of $J$
with $\si_1(f)=f$, choose one with minimal $h_2$-degree. As above,
$\si_2(f)=f$. That is, $f\in R^{\Z^2}$, the subalgebra
of invariants. However, $R^{\Z^2}=\C$. This
proves that $J$ contains a non-zero constant, so $J=R$. Hence $R$ is $\Z^2$-simple.

Now, note that we have $Y_1X_1=t_1=h_2-h_1$ and $X_1Y_1=\si_1(t_1)=h_2-h_1+1$,
hence $[-Y_1,X_1]=1$. Similarly we get $[Y_2,X_2]=1$. Since $\mu_{ij}=1$
for all $i,j$, we have $[Y_j,X_i]=0$ if $i\neq j$.
Thus, if $x\in\mathcal{X}$ has degree $g=(g_1,g_2)\in\Z_{>0} \times \Z_{>0}$ then
we have
\begin{equation}
1=\frac{1}{g_1!g_2!}(\ad -Y_1)^{g_1}\big((\ad Y_2)^{g_2}(x)\big)
\end{equation}
where $(\ad a)(b)=[a,b]=ab-ba$.
Thus $AxA=A$ for each $x\in\mathcal{X}$, where $A=\mathcal{S}(1,u,1)$.

For any $g=(g_1,g_2)\in \Z^2$ we have $\si_g(h_i)=h_i-g_i$ for $i=1,2$.
Thus it is clear that $\si:\Z^2\to\Aut_\K(R)$ is injective.
By Theorem \ref{CommutantDescription}, $C_A(R)=R$. Since we
always have $Z(A)\subseteq C_A(R)$, we get $Z(A)\subseteq R$. Hence, by Theorem \ref{thm:Simplicity_of_TGWAs}, $A$ is simple.
\end{proof}

\end{exmp}

\begin{exmp}
The construction of the following TGWA was first mentioned
in \cite[Example~1.3]{MT_def} and has been generalized to
higher rank cases in \cite{H09,Sergeev}.
Let 
 $n=2$, $R=\K[H]$, $\si_1(H)=H+1$, $\si_2(H)=H-1$,
 $t_1=H$, $t_2=H+1$, $\mu_{12}=\mu_{21}=1$ and
let $A=A\TGWdat$ be the associated TGWA.
One may verify that the consistency relations \eqref{eq:tgwa_consistency}
hold and that $A$ is $\K$-finitistic of type $A_2$.
It was shown in \cite[Example~6.3]{H09} that $A$ is isomorphic to the $\K$-algebra
with generators $X_1,X_2,Y_1,Y_2,H$ and defining relations
\begin{gather*}
\begin{alignedat}{2}
X_1H&=(H+1)X_1, &\qquad  X_2H&=(H-1)X_2, \\
Y_1H&=(H-1)Y_1, &\qquad Y_2H&=(H+1)Y_2, \\
Y_1X_1&=X_2Y_2=H, &\qquad   Y_2X_2&=X_1Y_1=H+1,\\
X_1Y_2&=Y_2X_1,&\qquad X_2Y_1&=Y_1X_2,
\end{alignedat}\qquad
\begin{aligned}
X_1^2X_2-2X_1X_2X_1+X_2X_1^2&=0,\\
X_2^2X_1-2X_2X_1X_2+X_1X_2^2&=0,\\
Y_1^2Y_2-2Y_1Y_2Y_1+Y_2Y_1^2&=0,\\
Y_2^2Y_1-2Y_2Y_1Y_2+Y_1Y_2^2&=0.
\end{aligned}
\end{gather*}
Using these relations it is easy to check that $0\neq X_1X_2-X_2X_1\in Z(A)$.
Thus, by Theorem \ref{thm:Simplicity_of_TGWAs}, $A$ is not simple.
\end{exmp}

\section{Comparison between TGWAs and strongly graded rings}

In \cite[Theorem 3.5]{Oinert} it was proved that if $S$ is a skew group ring with commutative neutral component $S_e$,
then $S_e$ is maximal commutative if and only if $S_e\cap I\neq \{0\}$ for all non-zero ideals
 $I\subseteq S$ (i.e. $S_e$ is an essential subalgebra of $S$).
For a TGWA $A=A\TGWdat$, maximal commutativity of $R(=A_0)$ in $A$ implies that
$R$ is an essential subalgebra of $A$, by Theorem \ref{IdealIntersection}. However, the converse is not true because
of Example \ref{ex:muSimple} which is a simple TGWA (hence $R$ trivially is essential)
but $R$ is not maximal commutative in $A$.

Furthermore, in \cite[Theorem 6.13]{Oinert} it was shown that if $S$ is a $G$-graded skew group ring with commutative neutral component $S_e$, then $S$ is simple if and only if $S_e$ is $G$-simple and maximal commutative in $S$. For a TGWA $A=A\TGWdat$, simplicity of $A$ implies that $R$ is $\Z^n$-simple but does not imply that $R$ is maximal commutative in $A$ (Example \ref{ex:muSimple} again). The converse need not hold either, by Example \ref{ex:nonsimpleGWA}.

Another result, \cite[Theorem 6.6]{Oinert}, states that if $S$ is a strongly $G$-graded ring such that the neutral component $S_e$ is maximal commutative in $S$, then $S$ is simple if and only if $S_e$ is $G$-simple. For a TGWA $A=A\TGWdat$, if $R$ is $\Z^n$-simple and maximal commutative in $A$ then still $A$ need not be simple, see Example \ref{ex:nonsimpleGWA}. However, by Lemma \ref{SimpleImpliesZnSimple}, if $A\TGWdat$ is a simple TGWA (even without $R$ being maximal commutative) then $R$ is $\Z^n$-simple.

Finally, \cite[Proposition 6.5]{Oinert} shows that if $S$ is a strongly $G$-graded ring with commutative $S_e$ such that $C_S(S_e)$ is $G$-simple, then $S$ is simple. For TGWAs, the condition that $C_A(R)$ is $\Z^n$-simple could be interpreted as saying that $C_A(R)$ has no non-zero proper ideals $J$ such that $X_iJ=JX_i$ and $Y_iJ=JY_i$ for all $i$. In any case this need not hold for TGWAs because of Example \ref{ex:nonsimpleGWA}.

\bibliographystyle{amsalpha}

\end{document}